\documentclass[11pt,a4paper]{amsart}
\usepackage{amssymb,amsfonts,amsmath,mathrsfs}
\usepackage{dsfont}
\usepackage{color}
\usepackage[colorlinks=true,linkcolor=blue,citecolor=red, pdfborder={0 0 0}]{hyperref}
\allowdisplaybreaks

\numberwithin{equation}{section}
 \newtheorem{prop}{Proposition}[section]
\newtheorem{theorem}[prop]{Theorem}
\newtheorem{cor}[prop]{Corollary}
\newtheorem{lemma}[prop]{Lemma}

\theoremstyle{definition}
\newtheorem{Def}[prop]{Definition}

\newtheorem{rem}[prop]{Remark}

\def\Bbb{\mathbb}
\def\d{\partial}
\def\dist{\operatorname{dist}}
\def\R{\Bbb R}
\def\1{\mathds{1}}
\def\T{\Bbb T}
\def\N{\Bbb N}
\def\Dx{\Delta_x}
\def\Nx{\nabla_x}
\def\Dt{\partial_t}
\def\({\left(}
\def\){\right)}
\def\eb{\varepsilon}
\def\k{\kappa}

\def\Cal{\mathcal}

\def\a{\alpha}
\def\E{\mathcal{E}}

\begin{document}
\title[Quintic  wave equations]{Infinite energy solutions for weakly damped quintic
 wave equations in $\R^3$}
\author[Mei, Savostianov, Sun, and Zelik] {Xinyu Mei${}^1$, Anton Savostianov${}^2$,
 Chunyou Sun${}^1$, and  Sergey Zelik${}^{1,3}$}

\begin{abstract}
The paper gives a comprehensive study of infinite-energy solutions and their long-time behavior
 for semi-linear weakly damped wave equations in $\R^3$ with quintic nonlinearities. This study includes
  global well-posedness of the so-called Shatah-Struwe solutions, their dissipativity, the existence
  of a locally compact global attractors (in the uniformly local phase spaces) and their extra regularity.
\end{abstract}


\thanks{This work is partially supported by  the RSF grant   19-71-30004  as well as  the EPSRC
grant EP/P024920/1 and NSFC grants No. 11471148, 11522109, 11871169.}

\address{${}^1$ School of Mathematics and Statistics, Lanzhou University,
Lanzhou \newline 730000,
P.R. China}
\email{meixy13@lzu.edu.cn}
\email{sunchy@lzu.edu.cn}

\address{${}^2$ Uppsala University, Uppsala, Department of Mathematics, Uppsala, 75106, Sweden.}
\email{anton.savostianov@math.uu.se
}
\address{${}^3$ University of Surrey, Department of Mathematics, Guildford, GU2 7XH, United Kingdom.}
\email{s.zelik@surrey.ac.uk}

\maketitle
\tableofcontents
\section{Introduction}
We study the following weakly damped wave equation:
\begin{equation}\label{eq.qdw}
\Dt^2 u+\gamma\Dt u+(1-\Dx)u+f(u)=g(t),\ \ \
\{u,\Dt u\}\big|_{t=0}=\{u_0,u_0'\}
\end{equation}
in a whole space $\R^3$.
Here $u(t,x)$ is the unknown function, $\Dx$ is the  Laplacian with respect to variable $x$,
 $\gamma$ is a  positive constant, $f:\R\to\R$ is a given non-linearity which is assumed to
 be of \emph{quintic} growth ($f(u)\sim u^5$) and to satisfy some natural conditions
 (stated in \eqref{4.f}) and $g$ belonging to the space $L^1_{loc}(\R_+,L^2_{loc}(\R^3))$ or its closed
  subspace $L^1_b(\R_+,L^2_b(\R^3))$, see Section \ref{s.sp} for definitions of key functional spaces.
\par
Dispersive or/and dissipative semilinear wave equations of the form \eqref{eq.qdw}  model various
 oscillatory processes
in many areas of  modern mathematical physics including  electrodynamics,
quantum mechanics, nonlinear elasticity,  etc. and are of a big permanent interest,
see \cite{lions,BV,Te,CV,straus,tao,sogge} and references therein.
\par
It is believed that the analytic properties and the dynamics as $t\to\infty$ of
solutions for damped wave equations \eqref{eq.qdw} strongly depend on the growth rate
of the non-linearity $f(u)$ as $u\to\infty$. Indeed, in the most studied case of cubic and
sub-cubic growth rate, the control of the energy norm which follows from the basic energy identity
is sufficient to get the well-posedness of the problem in a natural energy space,
dissipativity and further regularity of solutions as well as to develop the corresponding attractors
theory in both autonomous and non-autonomous cases as well as in bounded and unbounded domains,
see \cite{HR,BV,CV,hale,Lad,lions,MirZel,Te,ZCPAA2004} and references therein.
\par
We recall that the standard energy identity
\begin{multline}\label{0.energy}
 E(\xi_u(t))- E(\xi_u(\tau))=-\gamma\int_\tau^t\|\Dt u(s)\|^2_{L^2}\,ds+\\+
\int_\tau^t(\Dt u(s),g(s))\,ds, \ \ \xi_u(t):=\{u(t),\Dt u(t)\}
\end{multline}
can be formally obtained by multiplying equation \eqref{eq.qdw} by $\Dt u$
and integrating over $t$ and $x$. Here
$$
E(\xi_u):=\frac12\(\|\Dt u\|^2_{L^2}+\|\Nx u\|^2_{L^2}+\|u\|^2_{L^2}+2(F(u),1)\),
$$
$F(u):=\int_0^u f(z)\,dz$ and $(u,v):=\int_{\R^3}u(x)v(x)\,dx$.
This identity motivates the natural choice of the energy phase space and the class
of energy solutions (as the solutions for which the energy functional is finite) and also
 gives the control of the energy norm of the solution. Namely, if the non-linearity has a
 sub-quintic or quintic growth rate, due to the Sobolev embedding theorem
 $H^1\subset L^6$, the energy space is given by $\E:=H^1(\R^3)\times L^2(\R^3)$ and
 in the supercritical case $f(u)\sim u|u|^q$ with $q>4$, we need to take
  $\E:=(H^1(\R^3)\cap L^{q+2}(\R^3))\times L^2(\R^3)$
 in order to guarantee the finiteness of the energy functional.
\par
The case of super-cubic but sub-quintic growth rate ($2<q<4$) is a bit more complicated
 since the well-posedness of energy solutions is still an open problem here
  (at least in the case of bounded domains). However, this problem can be overcome using slightly
  more regular solutions than the energy ones for which, say, the mixed $L^4(0,T;L^{12}(\R^3))$
  space-time norm is finite for every $T>0$. These are  the so-called Shatah-Struwe
   (or Strichartz) solutions for which the well-posedness is known. The existence of such
   solutions is strongly based on the Strichartz estimates for the {\it linear} wave equation
     which are now available
    not only for the whole space $\R^3$ or the torus $\T^3$, but also for
    bounded domains with Dirichlet
    or Neumann boundary conditions, see \cite{Chem,Sogge2009,plan1,plan2,straus,St,tao}.
    Moreover, crucial for the attractor theory is the following energy-to-Strichartz estimate
    for such solutions
\begin{equation}\label{0.es}
\|u\|_{L^4(t,t+1;L^{12})}\le Q(\|\xi_u(t)\|_{\E})+Q(\|g\|_{L^1(t,t+1;L^2)}),
\end{equation}
where $Q$ is monotone increasing function which is independent of $t$ and the solution $u$.
 In the sub-quintic case this estimate is a straightforward corollary of the linear
 Strichartz estimate and perturbation arguments. Energy-to-Strichartz estimate \eqref{0.es}
 allows us to deduce the control and establish the dissipativity of $u$ in the Strichartz norm
 based on the standard energy estimate. Since the control of this norm is enough for the uniqueness,
 the obtained control gives the well-posedness, dissipativity and the existence of
 global/uniform attractors in the way which is similar to the clasical cubic case,
  see \cite{feireisl},\cite{kap4} and \cite{KSZ} for the case of $\R^3$, $\T^3$ and a
  bounded domain endowed with the Dirichlet boundary conditions respectively
   (see also \cite{SADE} for the case of damped wave equations with fractional damping).
\par
In contrast to this, very few is known about the solutions of \eqref{eq.qdw}
in the supercritical (superquintic) growth rate of the non-linearity $f$.
In this case the situation is somehow close to 3D Navier-Stokes problem, namely, we have the global
 existence of weak energy solutions for which the uniqueness is not known
  and the local existence of more regular solutions for which we do not know the
  global existence. It is expected that smooth solutions may blow up in finite time even in
  the defocusing case, but to the best of our knowledge there are no such examples. In this case
   the existing attractor theory  is related to multilavued semigroups or/and the so-called
   trajectory dynamical  systems and trajectory attractors,
   see \cite{CV,CV1,MirZel,ZelDCDS} (see also references therein).
\par
We now turn to the most interesting  borderline case of critical quintic non-linearity $f$ which is the
 main object of our study in this paper. In this case, the energy-to-Strichartz estimate \eqref{0.es} does
 not follow any more from the Strichartz estimate for
 the linear equation (at least in a straightforward way), so the proof of global existence for Shatah-Struwe
 solutions is usually based on the so-called non-concentration arguments and Pohozhaev-Morawetz
 equality, see \cite{Chem,Grill,kap1,kap2,kap3,SS1,SS,sogge,tao} (see also \cite{plan1,plan2}
for the case of bounded domains with Dirichlet or Neumann boundary conditions). This approach allows us
 to construct a Shatah-Struwe solution
 $u$ such that the $L^4(0,T;L^{12})$-norm is finite for all $T$, but does not allow to get any
  control of this norm through
the energy norm or to verify that the Strichartz norm does not grow as $T\to\infty$.
This is clearly not sufficient
 for the attractors. Indeed, without the uniform control of the Strichartz norm as $T\to\infty$,
 this extra regularity
 may a priori be lost in the limit and the attractor may contain the solutions which are less regular
than the Shatah-Struwe ones (for which we do not have the uniqueness theorem). Thus,  the uniform control
 of the Strichartz norm is crucial for the attractor theory.
\par
This problem has been overcome in \cite{KSZ} where the asymptotic regularity and existence of
global attractors
 for {\it autonomous} quintic wave equations in bounded domains of $\R^3$ has been established.
  The method suggested there is
heavily based on the existence of global Lyapunov function and on the related convergence of the
 trajectories to the set of equilibria and, by this reason cannot be extended to the
 non-autonomous case or to the case of infinite-energy solutions.
\par
An alternative method of verifying the asymptiotic smoothing property for the quintic
 wave equation \eqref{eq.qdw} has been recently suggested in \cite{SZ-UMN}. This method
  is based on a proper generalization of a direct energy-to-Stirichartz estimate for the
  model quintic wave equation
\begin{equation}\label{0.qmod}
 \Dt^2u-\Dx u+u^5=0
\end{equation}
 in $\R^3$ which in turn has been obtained earlier in \cite{BG1999} (see also \cite{tao1})
  via the profile decomposition  technique. This method allowed us (in \cite{SZ-UMN}) to build up more
   or less complete attractors theory for weakly damped quintic wave equation \eqref{eq.qdw}
    with {\it periodic} boundary conditions in both autonomous and non-autonomous cases. Note that
     this result cannot be extended to the case of Dirichlet or Neumann boundary conditions
      since the analogue of energy-to-Strichartz estimates for equation \eqref{0.qmod} is still an
       open problem for this case.
\par
In the present paper, which can be considered as a continuation of \cite{SZ-UMN}, we give a detailed
 study of the case where equation \eqref{eq.qdw} is considered in the whole space $x\in\R^3$. Note
  first of all that the finite-energy case $\xi_u(t)\in\E$ can be treated exactly as in \cite{SZ-UMN}
   and, by this reason, is not very interesting. The only difference is that,
   due to the non-compactness of Sobolev's embedding $H^1(\R^3)\subset L^2(\R^3)$, the sole
   asymptotic smoothing property will not give the asymptotic compactness (which is crucial for
    the existence of the attractor) and should be combined with the so-called tail estimates,
    see \cite{EZ2001,MirZel} and references therein for more details.
\par
However, the assumption that $\xi_u(t)\in\E$ is a big restriction since it assumes implicitly
that the solution $u(t,x)$ should decay sufficiently fast as $|x|\to\infty$, so many physically
 relevant solutions (such as homogeneous equilibria, space or space-time periodic  or/and quasi-periodic
 patterns as well as all solutions bifurcating from them) are automatically out of consideration.
 In addition, the extra conditions which we need to pose in order to get tail estimates are
  also restrictive and, in particular, for natural non-linearities like $f(u)=u|u|^q-\kappa u$,
   a global attractor in $\E$ does not exist if $\kappa>1$.
   \par
 By these reasons, it is natural, following \cite{BV1,F,MiS95,MirZel,Zhyp} (see also references therein),
 to consider {\it infinite} energy solutions for which $\xi_u(t)\in\E_{loc}$ only, in other words, only the
  restrictions of
 $\xi_u(t)$ to bounded domains should have finite energy and the total energy may be infinite.
 In this case the key energy equality makes no sense any more (the energy is infinite) and a
  number of extra difficulties arises. We note from the very beginning that these difficulties
   are not only technical, for instance, in contrast to the case of finite energy, the corresponding
    attractors usually have infinite Hausdorff and fractal dimensions and infinite topological entropy,
     so principally new types of limit dynamics appear, see \cite{Z-MMO} for more details.
\par
We will overcome the problem with infinite total energy by localizing the energy estimates using
 the machinery of weighted and uniformly local energy estimates, see \cite{ACD,EZ2001,Zhyp,MirZel}, as well as
  the  finite speed of propagation property which is the fundamental property of
  wave equations, see e.g., \cite{sogge},
   and which allows us to reduce the well-posedness result to the case of finite-energy solutions. This leads
    to our first main result.

\begin{theorem}\label{Th0.main1} Let the non-linearity $f$ satisfy some
 natural assumptions (see \eqref{4.f}),
$\xi_u(0)\in\E_{loc}$ and $g\in L^1_{loc}(\R_+,L^2_{loc}))$. Then, problem \eqref{eq.qdw} possesses
 a unique global Shatah-Struwe solution $u$ such that $\xi_u(t)\in \E_{loc}$ for all $t\ge0$ and, in addition,
 \begin{equation}
u\in L^4_{loc}(\R_+,L^{12}_{loc}).
 \end{equation}
\end{theorem}

We note that the local energy and Strichartz norms of $u$ can be estimated by the proper norms
 of the initial data and the external force $g$. However, these norms may grow in time if no
  extra assumptions on the growth of initial data and $g$ as $|x|\to\infty$ are posed, so we need
   to put extra restrictions if we want to speak about dissipativity and attractors. The natural
    choice of phase spaces for this is given by the so-called {\it uniformly local} phase spaces.
    The rigorous definitions of them will be given in Section \ref{s.sp} below and here we just
     mention that the uniformly local phase space $L^p_b(\R^3)$  consists of functions
      from $L^p_{loc}(\R^3)$ for which the following norm is finite:
      $$
      \|u\|_{L^p_b}:=\sup_{x_0\in\R^3}\|u\|_{L^p(B^1_{x_0})},
      $$
where $B^R_{x_0}$ stands for a ball of radius $R$ in $\R^3$ centered in $x_0$. The uniformly local version of
 Sobolev spaces and the energy space $\E_b$ are defined analogously.
\par
Our next result gives the dissipativity of the Shatah-Struwe solutions in uniformly local energy spaces.

\begin{theorem}\label{Th0.main2} Let the assumptions of Theorem \ref{Th0.main1} hold and let, in addition,
$\xi_u(0)\in\E_b$ and $g\in L^1_b(\R_+,L^{2}_b)$. Then the solution $u(t)$ constructed in Theorem
\ref{Th0.main1} belongs to $\E_b$ for all $t\ge0$ and possesses the following dissipative estimate:
\begin{equation}\label{0.dis}
\|\xi_u(t)\|_{\E_b}+\|u\|_{L^4(t,t+1;L^{12}_b)}\le
Q(\|\xi_u(0)\|_{\E_b})e^{-\beta t}+Q(\|g\|_{L^1_b(\R_+,L^2_b)}),
\end{equation}
where the positive constant $\beta$ and monotone function $Q$ are independent of $t$, $u$ and $g$.
\end{theorem}
The analogue of this estimate for the energy norm $\|\xi_u(t)\|_{\E_b}$ is well-known (see \cite{MS,Zhyp})
 and holds even in the case where $f$ has a super-critical growth rate, so the main novelty
 of \eqref{0.dis} is exactly the dissipative control of the Strichartz norm which is crucial
  for the uniqueness and attractors.
\par
We now turn to the attractors. For simplicity, we restrict ourselves to the autonomous case only
\begin{equation}\label{0.gaut}
g(t)\equiv g\in L^2_b(\R^3).
\end{equation}
In this case, thanks to Theorem \ref{Th0.main2}, the solution operators $S(t):\E_b\to\E_b$ defined via
$$
S(t)\xi_0=\xi_u(t), \ \xi_0\in\E_b,
$$
where $\xi_u(t)$ is a Shatah-Struwe solution of \eqref{eq.qdw} with the initial
condition $\xi_u\big|_{t=0}=\xi_0$ generate, a dissipative semigroup in the phase space $\E_b$ and we
 may speak about its global attractor. We recall that, in contrast to the case of bounded domains, a compact
  global attractor usually does not exists even in the simplest cases if we work in uniformly local spaces, so the so-called {\it locally}
   compact global attractor is used instead, see \cite{MirZel} and also Section \ref{s6} below. By definition, a locally
   compact global attractor is a bounded closed set in $\E_b$ which is compact in $\E_{loc}$ only, strictly
   invariant and attracts the images of bounded sets in $\E_b$ also in the topology of $\E_{loc}$ only.
\par
The next theorem can be considered as the third main result of the paper.

\begin{theorem}\label{Th0.main3} Let the assumptions of Theorem \ref{Th0.main2} hold and
 let, in addition, \eqref{0.gaut} is satisfied. Then, the solution semigroup $S(t):\E_b\to\E_b$ associated with
  equation \eqref{eq.qdw} possesses a locally compact global attractor $\Cal A$ in $\E_b$. This attractor
   is a bounded set of $\E^1_b:=H^2_b(\R^3)\times H^1_b(\R^3)$. Moreover, if the initial
   data $\xi_u(0)\in\E_b^1$, then $\xi_u(t)\in\E^1_b$ for all $t\ge0$ and the following estimate holds:
\begin{equation}\label{0.dis1}
\|\xi_u(t)\|_{\E_b^1}\le
Q(\|\xi_u(0)\|_{\E_b^1})e^{-\beta t}+Q(\|g\|_{L^2_b)}),
\end{equation}
where the positive constant $\beta$ and monotone function $Q$ are independent of $t$,
 $u$ and $g$. In other words,
problem \eqref{eq.qdw} is globally well-posed and dissipative in $\E^1_b$ as well.
\end{theorem}
As usual, the proof of this theorem is based on a decomposition of a solution $u(t)=v(t)+w(t)$, where $v(t)$
 is exponentially decaying and $w(t)$ is more regular and bootstrapping arguments. Similarly to \cite{SZ-UMN},
 we establish the extra regularity $w(t)\in\E^{\alpha}_b$ with $\alpha\in(0,\frac25]$ at the first step.
  And jump from
  $\E^{\alpha}_b$ to $\E_b^1$ at the second step. Although our proof follows in general
   the scheme suggested in
   \cite{SZ-UMN}, there are essential new difficulties here related with localization of
    Kato-Ponce type inequalities and the old scheme does not work directly. To overcome this difficulty, we
    introduce a new scheme of splitting $u(t)=\tilde v(t)+\tilde w(t)$ of the solution $u$ into a small and regular
     components which has an independent interest, see Remark \ref{Rem5.igo-go} for the details.
\par
The paper is organized as follows.
\par
 Section \ref{s.sp} gives an overview of weighted and uniformly
 local Sobolev spaces which are used in the paper. A special attention is paid to the localization of
 fractional Lebesgue-Sobolev spaces (=Bessel potential spaces) which are necessary for estimating
  the fractional norms of the differences $f(u)-f(v)$ via the Kato-Ponce inequality. Some commutator
   estimates which are necessary to treat these spaces are proved in Appendix \ref{sA}.
\par
The energy and Strichartz estimates for the linear equation \eqref{eq.qdw} (with $f=0$) which are
necessary for our study of the non-linear case are collected in Section \ref{s.lin}.
\par
Well-posedness and dissipativity of  quintic wave equation \eqref{eq.qdw} is studied in Section \ref{s4}.
The proofs of Theorems \ref{Th0.main1} and \ref{Th0.main2} are also given there.
\par
Decomposition of a solution $\xi_u(t)\in\E_b$ into exponentially decaying and more regular
 (bounded in $\E^\alpha_b$, $\alpha\le\frac25$) parts is verified in Section \ref{s5}. This is the
  most difficult part in the proof of Theorem \ref{Th0.main3}. Some estimates for the fractional norms
  of the difference $f(u)-f(v)$ are collected in Appendix \ref{sB}.
\par
Finally, the existence and $\E^1_b$ regularity of a locally compact global attractor for the considered equation \eqref{eq.qdw}
 is established in Section \ref{s6}. At the end of this section we also discuss briefly some corollaries
  of the proved Theorem \ref{Th0.main3} as well as its possible generalizations including entropy estimates, exponential attractors
   and extensions to the non-autonomous case.

\section{Weighted and uniformly local spaces}
\label{s.sp}

In this section we introduce a family of weighted and uniformly local  Sobolev spaces
which will be used throughout of the paper and briefly discuss useful relations between them,
 see e.g. \cite{EZ2001,MirZel,ZCPAM}
for more detailed exposition. We start by introducing the class of admissible weight functions and the
 corresponding weighted Lebesgue spaces.

\begin{Def}\label{def fsexpg}
Let $\mu >0$ be arbitrary. A function $\phi\in L^\infty_{loc}(\R^n)$ to be called a
weight function of exponential growth $\mu$ iff $\phi(x)>0$ and there holds inequality
\begin{equation} \label{2.1}
\phi(x+y)\leq C_\phi e^{\mu |y|}\phi(x),
\end{equation}
for every $x,y\in\R^n$. Let $\phi$ be a weight of an exponential growth. Then the norm in the weighted Lebesgue space
$L^p_\phi(\R^n)$, $1\le p\le\infty$ is defined via
\begin{equation}
\|u\|_{L^p_\phi}:=\(\int_{\R^n}\phi^p(x)|u(x)|^p\,dx\)^{1/p}.
\end{equation}
The uniformly local analogue $L^p_{b,\phi}(\R^n)$ is defined by the following norm:
\begin{equation}
\|u\|_{L^p_{b,\phi}}:=\sup_{x_0\in\R^n}\left\{\phi(x_0)\|u\|_{L^p(B^1_{x_0})}\right\},
\end{equation}
where $B^R_{x_0}$ stands for a ball of radius $R$ in $\R^n$ centered at $x_0$. We will
 write $L^p_b$ instead of $L^p_{b,1}$. The Sobolev spaces $W^{l,p}_\phi(\R^n)$ (resp. $W^{l,p}_{b,\phi}(\R^n)$)
  for $l\in\Bbb N$ are defined as spaces of distributions whose derivatives up to order $l$
  belong to $L^p_\phi(\R^n)$ (resp. $L^p_{b,\phi}(\R^n)$).
\end{Def}
\begin{rem}\label{rem 2.1}
One can easily check that if function $\phi$ is of exponential growth $\mu$
then so is the function $1/\phi$ with the same constant $C_\phi$. In other words \eqref{2.1} implies
\begin{equation}
\phi(x+y)\geq C_\phi^{-1}e^{-\mu|x|}\phi(y),
\end{equation}
for every $x,y \in\R^n$. It is also not difficult to see that a sum and a product of two weights of
 exponential growth is also a weight of an exponential growth, see \cite{EZ2001} for details.
\end{rem}

The key examples of weight functions of exponential growth are
$e^{-\eb|x-x_0|}$, its smooth analogue $e^{-\eb\sqrt{1+|x-x_0|^2}}$ and
 $(1+|x-x_0|^2)^\a$ where $\eb$ and $\a$ belong to $\R$.
 It is easy to see that the first two examples are functions of exponential growth
 $|\eb|$ and the last one is the weight function of exponential growth $\mu$ for
 arbitrary $\mu>0$. In particular,  the weights   $\phi_{\eb,x_0}(x)=e^{-\eb\sqrt{1+|x-x_0|^2}}$ possess
 an extra important property
 \begin{equation}\label{2.dw}
|D^k_x\phi_{\eb,x_0}(x)|\le C_k\eb^k\phi_{\eb,x_0}(x),\ x,x_0\in\R^n,
 \end{equation}
where $D^k_x$ stands for a collection of all partial derivatives of order
$k$ and the constant $C_k$ depends only on $k$. This property allows us to reduce the
 study of weighted spaces to non-weighted ones. Indeed, let us define the multiplication operator:
\begin{equation}
T_{\phi_{\eb,x_0}}u:=\phi_{\eb,x_0}u.
\end{equation}
Then, as a corollary of \eqref{2.dw}, we get the following result, see \cite{MirZel}.
\begin{prop}\label{Lem2.T} The operator $T_{\phi_{\eb,x_0}}$ realizes isomorphisms
 between the non-weighted space $W^{l,p}(\R^n)$ and its weighted analogue
  $W^{l,p}_{\phi_{\eb,x_0}}(\R^n)$ for any $l\in\Bbb N$. Moreover,
  $$
  \|T_{\phi_{\eb,x_0}}\|_{\mathcal L(W^{l,p}_{\phi_{\eb,x_0}},W^{l,p})}+
  \|T^{-1}_{\phi_{\eb,x_0}}\|_{\mathcal L(W^{l,p}, W^{l,p}_{\phi_{\eb,x_0}}, )}\le C_{l,p},
$$
where the constant $C_{l,p}$ is independent of $x_0$ and $\eb$ such that $|\eb|\le1$. The
 analogous result holds also for the spaces $W^{l,p}_{b,\phi_{\eb,x_0}}$ and $W^{l,p}_b$.
\end{prop}
The next standard proposition gives more convenient equivalent norms in weighted and uniformly local spaces.
\begin{prop}\label{Prop2.eq} Let $\phi(x)$ be a weight function of an exponential growth rate $\mu$
and let $\eb>\mu$. Then,
for every $u\in L^p_\phi(\R^n)$, $1\le p<\infty$, the following estimate holds:
\begin{equation}\label{2.weightint}
C_1\|u\|^p_{L^p_{\phi}}\le \int_{\R^n}\phi(x_0)^p \|u\|^p_{L^p_{\phi_{\eb,x_0}}}\,dx_0\le
C_2\|u\|^p_{L^p_{\phi}},
\end{equation}
where the constants $C_1$ and $C_2$ depend only on $\mu$ and $\eb$ and are independent of $u$ and $\phi$.
Analogously, for every $u\in L^p_{b,\phi}(\R^n)$, we have
\begin{equation}\label{2.weightsup}
C_1\|u\|_{L^p_{b,\phi}}\le \sup_{x_0\in\R^n}\left\{\phi(x_0) \|u\|_{L^p_{\phi_{\eb,x_0}}}\right\}\le
C_2\|u\|_{L^p_{b,\phi}}.
\end{equation}
\end{prop}
The proof of these estimates can be found in \cite{EZ2001,ZCPAM}.
\par
Proposition \ref{Prop2.eq} gives us a machinery for verifying various regularity estimates for
 linear PDEs by reducing them to the analogous non-weighted ones. We illustrate it on the following classical
 example:
 \begin{equation}\label{2.lap}
(1-\Delta)u(x)=g(x),\ \ x\in\R^n.
 \end{equation}
\begin{cor}\label{Cor2.reg} Let $\phi$ be a weight function of sufficiently small exponential growth
 $\mu$ ($\mu\le\mu_0\ll1$) and let $g\in L^p_\phi(\R^n)$ for some $1<p<\infty$. Then equation \eqref{2.lap} possesses
  a unique solution $u\in W^{2,p}_\phi(\R^n)$ and the following estimate holds:
  \begin{equation}\label{2.regweight}
  \|u\|_{W^{2,p}_\phi}\le C_p\|g\|_{L^p_\phi},
  \end{equation}
where the constant $C$ depends on $p$ and on the constant $C$ from inequality \eqref{2.1}. Analogously, if
$g\in L^p_{b,\phi}(\R^n)$ then the solution $u\in W^{2,p}_{b,\phi}(\R^n)$ and
  \begin{equation}\label{2.regul}
  \|u\|_{W^{2,p}_{b,\phi}}\le C_p\|g\|_{L^p_{b,\phi}}.
  \end{equation}
\end{cor}
\begin{proof} We restrict ourselves to verifying the estimates only. The existence of a solutions
 can be obtained  using the standard approximation arguments.
\par
 {\it Step 1.} We start with the classical non-weighted
maximal regularity estimate for the solutions of the elliptic equation \eqref{2.lap}, namely,
\begin{equation}\label{2.non}
\|u\|_{W^{2,p}}\le C_p\|g\|_{L^p},
\end{equation}
see e.g., \cite{Tri}.
\par
{\it Step 2.} We get the analogue of \eqref{2.non} for the space $L^p_{\phi}$ with special
 weights $\phi=\phi_{\eb,x_0}$ for small $\eb>0$
and arbitrary $x_0\in\R^n$. To this end, we write $v=T_{\phi_{\eb,x_0}}u$ for the new variable $v$ which satisfies
the equation
\begin{multline}\label{2.per}
(1-\Dx)v-B_{\eb,x_0}v=T_{\phi_{\eb,x_0}}g:=\tilde g,\\ B_{\eb,x_0}v:=
2\phi_{\eb,x_0}\Nx\phi_{-\eb,x_0}\Nx v+\phi_{\eb,x_0}\Dx\phi_{-\eb,x_0}v.
\end{multline}
Then, according to Proposition \ref{Lem2.T}, it is enough to verify the non-weighted
 $(L^p,W^{2,p})$-estimate for equation \eqref{2.per}. On the other hand, due to estimate \eqref{2.dw},
 we have
\begin{equation}
\|B_{\eb,x_0}v\|_{L^p}\le C\eb\|v\|_{W^{1,p}},
\end{equation}
so for sufficiently small $\eb>0$, equation \eqref{2.per} is a small regular perturbation
 of equation \eqref{2.lap}, so the regularity estimate for this equation is an immediate
 corollary of \eqref{2.non}, namely,
 $$
 \|v\|_{W^{2,p}}\le \|\tilde g+B_{\eb,x_0}v\|_{L^p}\le \|\tilde g\|_{L^p}+ C\eb\|v\|_{W^{2,p}}
 $$
 and assuming that $\eb$ is small enough that $C\eb\le 1/2$ we get the desired estimate for $v$.
 Returning back to the variable $u$ (and using Lemma \ref{Lem2.T} again), we arrive at
 \begin{equation}\label{2.w}
 \|u\|_{W^{2,p}_{\phi_{\eb,x_0}}}\le C_p\|g\|_{L^p_{\phi_{\eb,x_0}}}.
 \end{equation}
{\it Step 3.} The case of arbitrary weight $\phi$. We essentially use that the constant $C_p$
in \eqref{2.w} is independent of $x_0\in\R$. Therefore, multiplying \eqref{2.w} by
 $\phi(x_0)$ (where the exponential
growth rate $\mu$ of the weight $\phi$ satisfies ($\mu<\eb$), taking $p$th power from both sides of the
obtained inequality, integrating over $x_0\in\R^n$ and using \eqref{2.weightint} we get
the desired estimate \eqref{2.regul}. Analogously, replacing integration by taking
supremum over $x_0\in\R^n$,
 we get the desired estimate \eqref{2.regul}. This finishes the proof of the corollary.
\end{proof}
\begin{rem} The scheme described above works not only for the Laplace equation, but for many other
types of equations (elliptic, parabolic, etc.), see \cite{MirZel} and references therein. It also works
 for obtaining higher regularity and regularity in fractional Sobolev spaces. We give above the
  detailed derivation of the simplest regularity estimate for the reader's convenience only and will use the
   analogous results in what follows without further explanations.
\end{rem}
The next useful estimate is actually a combination of \eqref{2.1} and Minkowski inequality.

\begin{cor} [see, e.g., \cite{EZ2001}]\label{th wsp1}
	Let $u\in L^p_\phi(\R^n)$, where $\phi$ is a weight function of exponential growth $\mu>0$.
Then for any $1\leq q\leq \infty$ and every $\eb>\mu$, the following estimate is valid:
\begin{equation}\label{2.2}
\(\int_{\R^n}\phi(x_0)^{pq}\(\int_{\R^n}\phi_{\eb,x_0}^p(x)|u(x)|^p\,dx\)^qdx_0\)^\frac{1}{q}
\leq C\|u\|^p_{L^p_{\phi_{\eb,x_0}}},
\end{equation}
where the constant $C$ depends only on $\eb$, $\mu$ and $C_\phi$
from \eqref{2.1}.
\end{cor}
Indeed, thanks to Minkowski inequality and \eqref{2.1},
\begin{multline}
\(\int_{\R^n}\(\int_{\R^n}\(\phi(x_0)\phi_{\eb,x_0}(x)|u(x)|\)^p\,dx\)^q\,dx_0\)^\frac{1}{q}\le\\\le
 \int_{\R^n}\(\int_{\R^n}\(\phi(x_0)\phi_{\eb,x_0}(x)|u(x)|\)^{pq}\,dx_0\)^{1/q}\,dx\le\\\le
 C_\phi^p\int_{\R^n}\phi(x)^p|u(x)|^p\(\int_{\R^n}e^{pq(\mu-\eb)|x-x_0|}\,dx_0\)^{1/q}\,dx=C\|u\|_{L^p_\phi}^p.
\end{multline}
The next proposition gives another  way to reduce the study of  weighted spaces
 to the non-weighted case.

\begin{prop}\label{th wsp2}
	Let  $\phi$ be a function of exponential growth $\mu$ and let $R>0$ be a fixed number.
Then for any $p\in[1;\infty)$ the following estimates are valid:
\begin{equation}\label{2.lpr}
C_1\|u\|_{L^p_\phi}^p\le \int_{\R^n}\phi^p(x_0)\|u\|^p_{L^p(B^R_{x_0})}\,dx_0\le
 C_2\|u\|^p_{L^p_\phi},
\end{equation}
where the constants $C_1$ and $C_2$ depend on $R$, $C_\phi$, $p$, and $\mu$ only.
\end{prop}
For the proof of this estimate, see e.g., \cite{EZ2001}.
\par
As an immediate corollary of this estimates, we get the equivalent norms in Sobolev spaces
$W^{l,p}_\phi(\R^n)$ for integer $l>0$.

\begin{cor}\label{cor wsp3}
Let $l\in\Bbb{N}$, $1\le p\le\infty$. Then an equivalent norm in
 $W^{l,p}_{\phi}(\R^n)$ is given by the following expression:
\begin{equation}\label{2.8}
\|u\|_{W^{l,p}_{\phi,R}}:=\(\int_{\R^n}\phi^p(x_0)\|u\|^p_{W^{l,p}(B^R_{x_0})}\,dx_0\)^\frac{1}{p}.
\end{equation}
	In particular we obtain that norms \eqref{2.8} are equivalent for different $R>0$.
\end{cor}

We see that representation \eqref{2.8} reduces weighted Sobolev norm to Sobolev norm on bounded domains.
 Particularly, this gives the benefit of using standard Sobolev embeddings theorems for bounded domains
  (see \cite{MirZel}). Moreover, in analogy to \eqref{2.8} we are able to define
fractional weighted Besov-Sobolev spaces. We recall that, for any domain $V$ with smooth
boundary and any $s>0$, $s\notin\N$, the space $W^{s,p}(V)=B^s_{p,p}(V)$ is defined via the following norm:
\begin{equation*}
\|u\|^p_{W^{s,p}(V)}=\|u\|^p_{W^{[s],p}(V)}+
\sum_{|\alpha|=[s]}\int_{x\in V}\int_{y\in V}
\frac{|\d^\alpha u(x)-\d^\alpha u(y)|^p}{|x-y|^{n+\{s\}p}}dxdy,
\end{equation*}
where $[s]$ and $\{s\}$ denote integer and fractional part of $s$ respectively. As usual, for negative
non-integer $s$, the space $W^{s,p}(V)$ is defined by duality, see \cite{Tri} for the details.

\begin{Def}
	Let $s\in\R$ and  $1\le p\le\infty$ and $R>0$  be fixed numbers and let $\phi$ be a weight function
 with an exponential growth $\mu$.
   The equivalent norms in the space $W^{s,p}_\phi(\R^n)$ are defined by
\begin{equation}\label{2.9}
\|u\|_{W^{s,p}_{\phi,R}}:= \(\int_{\R^n} \phi^p(x_0)\|u\|^p_{W^{s,p}(B^R_{x_0})}dx_0\)^{1/p},
\end{equation}
where $R>0$ is arbitrary. We will write $\|u\|_{W^{s,p}_\phi}$ instead of $\|u\|_{W^{s,p}_{\phi,1}}$.
\end{Def}
 It is not difficult to check that norms defined by \eqref{2.9} are indeed equivalent for different $R>0$
  as well as \eqref{2.9}  gives usual
norm for $W^{s,p}(\R^n)$ if we take $\phi\equiv 1$ (see \cite{EZ2001}).
 Hence the above definition is natural. It is also straightforward to check that, analogously to
 \eqref{2.weightint},
\begin{equation}
\|u\|_{W^{s,p}_{\phi}}^p\sim \int_{\R^n}\phi(x_0)^p\|\phi_{\eb,x_0}u\|^p_{W^{s,p}(\R^n)}\,dx_0,
\end{equation}
if $\mu<\eb$,
 so the analogues of Proposition \ref{Lem2.T} and Corollary
\ref{Cor2.reg} hold for fractional weighted Besov-Sobolev spaces as well.

\begin{rem} It is useful to introduce the following notation for the above mentioned equivalent norms in
$W^{s,p}_b(\R^n)$:
$$
\|u\|_{W^{s,p}_{b,R}}:=\sup_{x_0\in\R^n}\|u\|_{W^{s,p}(B^R_{x_0})}.
$$
Then, the equivalence means that, for any $R_1,R_2>0$,
$$
C_{R_1,R_2}^{-1}\|u\|_{W^{s,p}_{b,R_1}}\le \|u\|_{W^{s,p}_{b,R_2}}\le C_{R_1,R_2}\|u\|_{W^{s,p}_{b,R_1}}
$$
for some positive constant $C_{R_1,R_2}$. In particular, the case $R_1=2R_2=R\ge1$ in especially
 interesting for us. Note that in this case the constant $C_{R_1,R_2}$ is actually independent
  of $R_1$ and $R_2$. The last fact can be easily verified using scaling arguments.
\end{rem}

We also need the scale of weighted Lebesgue-Besov spaces $H^{s,p}_\phi(\R^n)$ (or Bessel potential spaces).
 Recall that in the non-weighted case they are usually defined via the Fourier transform:
$$
H^{s,p}(\R^n):=\left\{u\in\mathcal S'(\R^n),\ \|u\|_{H^{s,p}}:=
\|\mathcal F^{-1}((1+|\xi|^2)^{s/2}\mathcal F u)\|_{L^p}<\infty\right\},
$$
where $\mathcal F$ is a Fourier transform, $s\in\R$ and $1<p<\infty$. Alternatively,
these spaces can be defined as domains of fractional powers of the operator $1-\Dx$ in $L^p$:
$$
H^{s,p}(\R^n)=D((1-\Dx)^{s/2}).
$$
It is well-known that $W^{s,p}(\R^n)=H^{s,p}(\R^n)$ if $p=2$ or $s\in\Bbb Z$. But for
non-integer $s\ge0$, we have the proper inclusion
$$
H^{s,p}(\R^n)\subset W^{s,p}(R^n) \ \ \text{if $p>2$}
$$
and the opposite proper inclusion if $p<2$, see \cite{Tri} for details.
\par
The space $H^{s,p}(V)$, where $V$ is a smooth bounded domain in $\R^n$
 (we will consider in this paper only the case $V=B^R_{x_0}$), is usually defined as a restriction of
 $H^{s,p}(\R^n)$ to $V$:
 $$
 H^{s,p}(V)=\left\{v\in \mathcal D'(V),\ \ \exists u\in H^{s,p}(\R^n),\ \ u\big|_{\Omega}=v\right\}
 $$
endowed with the standard factor-norm. It is also known that the restriction operator $u\to u\big|_{V}$
 is a retraction and the corresponding co-retraction (extension operator) can be chosen independently
  of $1<p<\infty$ and $|s|\le N$ for every fixed $N\in\Bbb N$, see \cite{Tri}. Mention also a useful relation
\begin{equation}\label{2.int}
H^{s,p}(V)=[L^p(V),W^{1,p}(V)]_s, \ 0<s<1,
\end{equation}
where $[\cdot,\cdot]_s$ means complex interpolation, see \cite{Tri}. Throughout of the paper  we
 will write below $H^s$ instead of $H^{s,2}$.
\par
The main reason for us to use fractional Lebesgue-Sobolev spaces is the following Kato-Ponce estimate
 which is crucial for obtaining the further regularity of solutions for the considered damped wave
  equation and which is naturally formulated exactly in these spaces.
\begin{prop}\label{Prop2.KP} Let $V$ be a bounded domain with smooth boundary and let $0<\alpha<1$
 and $1<r<\infty$. Then,
\begin{equation}\label{2.KP}
\|uv\|_{H^{\alpha,r}(V)}\le C\|u\|_{L^{p_1}(V)}\|v\|_{H^{\alpha,q_1}(V)}+C\|v\|_{L^{p_2}(V)}\|u\|_{H^{\alpha,q_2}(V)},
\end{equation}
where $\frac1r=\frac1{p_i}+\frac1{q_i}$,\ \ $1< p_i,q_i<\infty$.
\end{prop}
The proof of this estimate can be found, e.g., in \cite{Chem} for $V=\R^n$. The general case is reduced
 to the case $V=\R^n$ using the extension operator.
\begin{rem}\label{Rem2.cut} Mention also one more obvious, but useful property of the introduced norms.
 Namely, let $\psi_{x_0}\in C^\infty_0(\R^n)$ be a cut-off function such that
 $\psi_{x_0}(x)\equiv1$ for $|x-x_0|\le1$ and $\psi_{x_0}(x)\equiv0$ if $|x-x_0|\ge 3/2$. Then
 \begin{multline}\label{2.V}
\|u\|_{H^{s,p}(B^1_{x_0})}\le \|\psi_{x_0}u\|_{H^{s,p}(\R^n)}\le\\\le C\|\psi_{x_0}u\|_{H^{s,p}(B^2_{x_0})}
\le C\|u\|_{H^{s,p}(B^2_{x_0})}.
 \end{multline}
Indeed, these estimates follow in a straightforward way from the definition of the norm in $H^{s,p}(V)$
 and identity \eqref{2.int}.
\end{rem}
In order to introduce and study  weighted fractional Lebesgue-Sobolev spaces we need the following commutator estimate.
\begin{prop}\label{Prop2.com} Let $s\in(0,1)$, $1<p<\infty$ and $|\eb|$ be small enough. Then the following estimate holds:
\begin{equation}\label{2.com1}
\|\phi_{\eb,x_0}(1-\Dx)^{s/2}u-(1-\Dx)^{s/2}(\phi_{\eb,x_0}u)\|_{L^p}\le C_p|\eb|\|u\|_{L^p_{\phi_{\eb,x_0}}},
\end{equation}
where the constant $C_p$ depends on $p$ and $s$. Moreover, for any $\psi\in C^\infty_0(\R^n)$ and sufficiently small $\eb>0$,
\begin{equation}\label{2.com2}
\|\psi(1-\Dx)^{s/2}u-(1-\Dx)^{s/2}(\psi u)\|_{L^p}\le C_{p,\psi,\eb,x_0}\|u\|_{L^p_{\phi_{\eb,x_0}}}.
\end{equation}
\end{prop}
Although these estimates are more or less standard, we sketch the proof in Appendix \ref{sB} below.
\begin{cor}\label{Cor2.HT} Let $0<s<1$ and $|\eb|$ be small enough. Then
\begin{equation}\label{2.equ}
C_1\|\phi_{\eb,x_0}u\|_{H^{s,p}}\le\|\phi_{\eb,x_0}(1-\Dx)^{s/2}u\|_{L^p}\le
C_2\|\phi_{\eb,x_0}u\|_{H^{s,p}},
\end{equation}
where the constants $C_i$ are independent of $x_0\in\R^n$.
\end{cor}
Indeed, according to \eqref{2.com1},
\begin{equation*}
(1-C_p|\eb|)\|\phi_{\eb,x_0}u\|_{H^{s,p}}\le \|\phi_{\eb,x_0}(1-\Dx)^{s/2}u\|_{L^p}\le(1+C_p|\eb|)\|\phi_{\eb,x_0}u\|_{H^{s,p}}.
\end{equation*}
We are now ready to define the spaces $H^{s,p}_\phi(\R^n)$ and $H^{s,p}_{b,\phi}(\R^n)$.
\begin{Def} Let $s\in\R$ and $1<p<\infty$ and let $\phi$ be a weight of sufficiently small
 exponential growth rate $\mu$. Then the norms in the spaces $H^{s,p}_\phi(\R^n)$ and $H^{s,p}_{b,\phi}(\R^n)$ are
  defined by
\begin{equation}\label{2.wh}
\|u\|_{H^{s,p}_\phi}^p:=\int_{\R^n}\phi(x_0)^p\|u\|^p_{H^{s,p}(B^1_{x_0})}\,dx_0
\end{equation}
and
\begin{equation}\label{2.whb}
\|u\|_{H^{s,p}_{b,\phi}}:=\sup_{x_0\in\R^n}\left\{\phi(x_0)\|u\|_{H^{s,p}(B^1_{x_0})}\right\}
\end{equation}
respectively.
\end{Def}
\begin{cor}\label{Cor2.Heq} Let $s\in\R$, $1<p<\infty$ and let $\phi$ be a weight function of a sufficiently
small exponential growth $\mu$. Then, for any $R>0$ and sufficiently small $\eb>\mu$, we have
\begin{multline}\label{2.hcw}
  C_1\int_{\R^n}\phi(x_0)^p\|u\|^p_{H^{s,p}(B^R_{x_0})}\,dx_0\le\\\le
   \int_{\R^n}\phi(x_0)^p\|\phi_{\eb,x_0}u\|^p_{H^{s,p}(\R^n)}\,dx_0\le
    C_2\int_{\R^n}\phi(x_0)^p\|u\|^p_{H^{s,p}(B^R_{x_0})}\,dx_0,
\end{multline}
where the constants $C_1$ and $C_2$ may depend on $R$.
\end{cor}
\begin{proof} We give the proof for the case $0<s<1$ only (since we have $W^{s,p}=H^{s,p}$ for integer $s$,
the general case can be reduced to this particular one). Moreover, analogously to \eqref{2.V}, we have
\begin{equation}\label{2.cw}
\|u\|_{H^{s,p}(B^R_{x_0})}\le C\|\phi_{\eb,x_0}u\|_{H^{s,p}(B^R_{x_0})}\le
C\|\phi_{\eb,x_0}u\|_{H^{s,p}(\R^n)},
\end{equation}
so the left inequality of \eqref{2.hcw} is obvious. To prove the right inequality, we assume for simplicity
that $R=2$ and use \eqref{2.equ}, \eqref{2.com2} with $\psi=\psi_{x_0}$ and
 together with \eqref{2.lpr} to get
 \begin{multline}
 \int_{\R^n}\phi(x_0)^p\|\phi_{\eb,x_0}u\|^p_{H^{s,p}(\R^n)}\,dx_0\le\\\le
 C\int_{\R^n}\phi(x_0)^p\|\phi_{\eb,x_0}(1-\Dx)^{s/2}u\|^p_{L^p(\R^n)}\,dx_0\le\\\le
 C\int_{\R^n}\phi(x_0)^p\|(1-\Dx)^{s/2}u\|^p_{L^p(B^1_{x_0})}\,dx_0\\\le
 C\int_{\R^n}\phi(x_0)^p\|\psi_{x_0}(1-\Dx)^{s/2}u\|^p_{L^p(\R^n)}\,dx_0\\\le
 C\int_{\R^n}\phi(x_0)^p\|\phi_{\eb,x_0}u\|^p_{H^{s,p}(\R^n)}\,dx_0+
 C\int_{\R^n}\phi(x_0)^p\|u\|^p_{L^p_{\phi_{\eb,x_0}}}\,dx_0\\\le
 2C\int_{\R^n}\phi(x_0)^p\|u\|^p_{H^{s,p}(B^2_{x_0})}\,dx_0.
 \end{multline}
Here we have implicitly used that, due to \eqref{2.weightint} and \eqref{2.lpr},
$$
\int_{\R^n}\phi(x_0)^p\|u\|^p_{L^p_{\phi_{\eb,x_0}}}\,dx_0\le C'\|u\|_{L^p_\phi}^p\le
 C'\int_{\R^n}\phi(x_0)^p\|u\|^p_{L^p(B^2_{x_0})}\,dx_0
$$
and the corollary is proved.
\end{proof}
\begin{rem}
Estimates obtained in Proposition \ref{Prop2.com} and Corollaries \ref{Cor2.HT}
and \ref{Cor2.Heq} show that the results concerning embeddings and regularity in the weighted spaces
$H^{s,p}_\phi$ and $H^{s,p}_{b,\phi}$ can be obtained in the same way as for the spaces $W^{s,p}_\phi$ and $W^{s,p}_b$.
We will use this fact in a sequel without further details.
\par
Note also that the above mentioned scheme gives also the weighted
 Kato-Ponce estimate which has an independent interest:
\begin{equation}
\|uv\|_{H^{\alpha,r}_{\phi_1\phi_2}}\le C\|u\|_{L^{p_1}_{\phi_1}}\|v\|_{H^{\alpha,q_1}_{\phi_2}}+
C\|v\|_{L^{p_2}_{\phi_2}}\|u\|_{H^{\alpha,q_2}_{\phi_1}},
\end{equation}
where the exponents $\alpha,r, p_i,q_i$ are the same as in Proposition \ref{Prop2.KP} and $\phi_i$ are weights of sufficiently small
 exponential growth rate. We need not this result for what follows, so we leave
 its rigorous proof to the reader.
\end{rem}
We will systematically use in what follows the spaces of functions $u(t,x)$ which have different regularity
with respect to time $t$ and space $x$ variables, for instance $L^p(\R,L^q(\R^n))$, $L^p(\R,H^{s,q}(\R^n))$
 or/and their weighted and uniformly local analogue. In slight abuse of notations we denote by
$L^p(A,B;L^q_b(\R^n))$ and
 $L^p_b(A,B;L^q_b(\R^n))$. where $-\infty\le A<B\le\infty$, the spaces generated by the following norms:
$$
\|u\|_{L^p(A,B;L^q_b)}:=\sup_{x_0\in\R^n}\|u\|_{L^p(A,B;L^q(B^1_{x_0}))}
$$
and
 $$
 \|u\|_{L^p_b(A,B;L^q_b)}:=\sup_{x_0\in\R^n}\sup_{T\in[A,B]}\|u\|_{L^p(T,\min\{B,T+1\};L^q(B^1_{x_0}))}.
 $$
respectively. The spaces $L^p(0,T;H^{s,p}_b)$ and $L^p_b(\R,H^{s,p}_b)$ as
well as spaces $L^p_b(t,t+1;H^{s,p}_b)$ are defined analogously. Crucial is that the
supremum with respect to $x_0$ or/and $t\in\R_+$ is {\it always} taken
 {\it after} the integration in time. We will not consider other type of spaces in our paper.

\section{Linear wave equation: preliminaries and basic estimates}\label{s.lin}

In this section we give the weighted analogues of the regularity result for the following
 damped wave equation:
\begin{equation}\label{lin.eq}
\Dt^2v+\gamma\Dt v+(-\Dx+1)v=g(t),\ \ \xi_v\big|_{t=0}=\xi_0
\end{equation}
in the whole space $x\in\R^3$. Here and below $\xi_v$ stands for the pair of functions $v$ and $\Dt v$
 ($\xi_v:=\{v,\Dt v\}$). The initial data $\xi_0$ will be taken from the energy spaces
\begin{equation}\label{3.E}
\E^\alpha_{loc} := H^{1+\alpha}_{loc}(\R^3)\times H^\alpha_{loc}(\R^3),\ \alpha \in\R.
\end{equation}
or from their weighted and uniformly local analogues
 ($\E^\alpha_{\phi}$ and $\E^\alpha_b$ respectively). We will write $\E$ instead of $\E^0$.
\par
We will always assume here that $\gamma>0$ is a fixed constant and the external force $g$ satisfies
\begin{equation}\label{3.g}
g\in L^1_{loc}(\R_+,H^\alpha_{loc}(\R^3)).
\end{equation}
Let us start by recalling the classical energy estimate for solutions of \eqref{lin.eq} in the non-weighted
 case and $\alpha=0$.
\begin{prop}\label{Prop3.encl} Let $\xi_0\in\E$ and $g\in L^1_{loc}(\R_+,L^2(\R^3))$. Then problem \eqref{lin.eq} possesses
a unique solution $\xi_v\in C(\R_+,\E)$ and the following estimate holds:
\begin{equation}
\|\xi_v(t)\|^2_{\E}\le C\|\xi_v(0)\|^2_{\E}e^{-\beta t}+C\(\int_0^te^{-\beta(t-s)}\|g(s)\|_{L^2}\,ds\)^2
\end{equation}
for some positive constants $C$ and $\beta$ depending only on $\gamma$. Moreover, the function $t\to\|\xi_v(t)\|^2_{\E}$
is absolutely continuous and the following energy identity:
\begin{equation}\label{3.en}
\frac12\frac {d}{dt}\|\xi_v(t)\|^2_{\E}+\gamma\|\Dt v(t)\|^2_{L^2}=(g,\Dt v)
\end{equation}
holds
for almost all $t\in\R_+$. Here and below $(f,g):=\int_{\R^3}f(x)g(x)\,dx$ stands for
the standard inner product in $L^2(\R^3)$.
\end{prop}
For the proof of this result, see e.g., \cite{BV,Te}.
\par
The next technical tool is the so-called Strichartz estimates which are crucial
 for the study of the non-linear case.
\begin{prop}\label{Prop3.str} Under the assumptions of Proposition \ref{Prop3.encl} the solution $v$
satisfies the following estimate:
\begin{equation}\label{3.str}
\|v\|_{L^p(0,1;L^{\frac{6p}{p-2}})}\le C_p\(\|\xi_v(0)\|_{\E}+\|g\|_{L^1(0,1;L^2)}\)
\end{equation}
for all $p\in(2,\infty]$.
\end{prop}
For the proof of this estimate, see \cite{sogge,straus,tao}.
\begin{rem} The most important case for us is $p=4$ which gives $L^4(L^{12})$-estimate
 for the solution $v$. To control the nonlinearity we also need $p=5$ which however can be derived
  from $p=4$ and the energy estimate by using the following interpolation inequality:
  \begin{equation}\label{3.5int}
  \|v\|_{L^5(L^{10})}^5\le C\|v\|_{L^4(L^{12})}^4\|v\|_{L^\infty(L^6)}\le
   C\|v\|^4_{L^4(L^{12})}\|\xi_v\|_{L^\infty(\E)},
  \end{equation}
so we will state below the estimates for $p=4$ only.
\end{rem}
Combining Propositions \ref{Prop3.encl} and \ref{Prop3.str}, we get the following result.

\begin{cor}\label{th.lin.damp} Let $\xi_v(0)\in \E^\alpha$
 and $g\in L^1_{loc}(\R,H^\alpha(\R^3))$ for some $\alpha\in\R$. Then, for all $\beta\in (0,\beta_0]$,
the solution $v(t)$ of problem \eqref{lin.eq} possesses the following estimate:
\begin{multline}\label{lin.est}
	\|\xi_v(t)\|_{\E^\alpha}+\(\int_0^t e^{-4\beta(t-s)}\|v(s)\|^4_{H^{\alpha,12}}\,ds\)^{1/4}\le\\\le
 C\left(\|\xi_v(0)\|_{\E^{\alpha}}e^{-\beta t}+\int_0^te^{-\beta(t-s)}\|g(s)\|_{H^\alpha}\,ds\right),
\end{multline}
where the positive constants $C$ and  $\beta_0$ are independent of $t\geq 0$, $v$ and $g$.
\end{cor}
The proof of this estimate is straightforward and can be found in \cite{SZ-UMN}. We mention here only that the general case
 $\alpha\in\R$ is reduced to the case $\alpha=0$ by applying the operator $(1-\Dx)^{\alpha/2}$ to both sides
  of equation \eqref{lin.eq}.
  \par
We will also need the finite speed propagation estimate for solutions of \eqref{lin.eq}.
\begin{prop}\label{Prop3.fin} Let the assumptions of Proposition \ref{Prop3.encl} hold,
$x_0\in\R^3$ and $R\in\R_+$. Then the solution $v$ of problem \eqref{lin.eq}
satisfies the following estimate:
\begin{equation}\label{3.cone}
\|\xi_v(t)\|_{\E(B^{R-t}_{x_0})}\le C\|\xi_v(0)\|_{\E(B^R_{x_0})}+C\int_0^t\|g(s)\|_{L^2(B^{R-s}_{x_0})}\,ds,
\end{equation}
where $0\le t<R$ and the constant $C$ is independent of $R$,  $x_0$ and $t$.
\end{prop}
See e.g., \cite{sogge} for the proof of this estimate.

\begin{rem} Let us define a cone
$$
C^R_{x_0}=\{(t,x)\in\R_+\times\R^3\,, |x-x_0|\le R-t\}.
$$
Then, estimate \eqref{3.cone} shows that the values of $v\big|_{C^R_{x_0}}$ depend only on
the values of $\xi_v(0)\big|_{B^R_{x_0}}$
and the values of $g\big|_{C^R_{x_0}}$. In particular, if for two solutions $v_1,v_2\in C_{loc}(\R_+,\E_{loc})$
 of equation \eqref{lin.eq} and we know that
$$
 \xi_{v_1}(0)\big|_{B^R_{x_0}}=\xi_{v_2}(0)\big|_{B^R_{x_0}} \ \text{and}\ \ g_1\big|_{C^R_{x_0}}
 =g_2\big|_{C^R_{x_0}},
$$
then $v_1\big|_{C^R_{x_0}}=v_2\big|_{C^R_{x_0}}$. In particular, this property allows us to verify
 the existence and uniqueness of solution $\xi_v\in C_{loc}(\R_+,\E_{loc})$ (with the initial
 data $\xi_0\in\E_{loc}$ and $g\in L^1_{loc}(\R_+,L^2_{loc})$) using Proposition \ref{Prop3.encl}
 for square integrable case. We will use this idea in the non-linear case as well.
\end{rem}
We conclude the section by  the weighted analogue of estimate \eqref{lin.est}.
\begin{cor}\label{lin.th.astr} Let $\eb\in\R$ be a sufficiently small positive number, $x_0\in\R^3$ and
$\alpha\in[0,1]$.
Let also $\xi_0\in \E^\alpha_{\phi_{\eb,x_0}}$ and
$g\in L^1_{loc}(\R,H^\alpha_{\phi_{\eb,x_0}}(\R^3))$. Then the solution $v(t)$ of problem
 \eqref{lin.eq} possesses the following estimate:
\begin{multline}\label{3.al-weight}
\|\xi_v(t)\|_{\E^\alpha_{\phi_{\eb,x_0}}}+\(\int_0^te^{-4\beta(t-s)}
\|v(s)\|^4_{H^{\alpha,12}_{\phi_{\eb,x_0}}}\,ds\)^{1/4}\le\\
	 C\|\xi_0\|_{\E^\alpha_{\phi_{\eb,x_0}}}e^{-\beta t}+
C\int_0^te^{-\beta(t-s)}\|g(s)\|_{H^\alpha_{\phi_{\eb,x_0}}}\,ds,
	\end{multline}
	where positive constants $C$ and  $\beta$ are independent
 of $t\geq 0$ and $\xi$, $g$ and $\eb$.
\end{cor}
\begin{proof} We give the proof for the fractional case $\alpha\in(0,1)$ only. The case $\alpha=0$
 is much simpler and we leave the proof to the reader and the case $\alpha=1$ can be reduced to $\alpha=0$
  by differentiation of the equation in $x$.
\par
 We use the trick with isomorphism
$T_{\eb,x_0}: H^{\alpha,p}_{\phi_{\eb,x_0}}(\R^3)\to H^{\alpha,p}(\R^3)$, see estimate
\eqref{2.equ} and \eqref{2.weightint}, described in the proof of Corollary \ref{Cor2.reg}. Namely,
for proving the weighted estimate \eqref{3.al-weight}, it is sufficient to estanblish its
 non-weighted analogue \eqref{lin.est} for the function $V=\phi_{\eb,x_0}v$ which satisfies
 the perturbed analogue of \eqref{lin.eq}:
\begin{equation}\label{lin.eqper}
\Dt^2V+\gamma\Dt V+(-\Dx+1)V=\phi_{\eb,x_0}g(t)+B_{\eb,x_0}V,\ \ \xi_v\big|_{t=0}=\phi_{\eb,x_0}\xi_0,
\end{equation}
where the operator $B_{\eb,x_0}$ is the same as in \eqref{2.per} and, therefore, satisfies the estimate
$$
\|B_{\eb,x_0}V\|_{H^\alpha}\le C\eb\|V\|_{H^{1+\alpha}}.
$$
Thus, estimate \eqref{lin.est} for $V$ follows from the analogous estimate for $v$
 (treating the term $B_{\eb,x_0}V$ as a perturbation)
 if $\eb$ is small enough. This proves the corollary.
\end{proof}

\section{Quintic wave equation: well-posedness and dissipativity}\label{s4}

The aim of this section is to study the  infinite-energy solutions
 to the following semi-linear weakly damped wave equation:
\begin{equation}\label{eq.main}
\Dt^2u+\gamma\Dt u+(-\Dx+1)u +f(u)=g(t),\ \ \xi_u\big|_{t=0}=\xi_0
\end{equation}
in the whole space $x\in\R^3$. It is assumed that the nonlinearity $f\in C^2(\R)$ has quintic growth rate:
\begin{equation}\label{4.f}
f(u)=u^5+h(u),\ \ |h''(u)|\le C(1+|u|^q),\ \ h(0)=0
\end{equation}
for some exponent $0\le q<3$. We start with a general case where the  initial data  $\xi_0\in\E_{loc}$ and
$g\in L^1_{loc}(\R_+,L^2_{loc}(\R^3))$, so we do not pose up to the moment any restrictions on
 the growth of the solution as $|x|\to\infty$.

\begin{Def}\label{Def4.sol} A function $u(t)$ such that $\xi_u(t)\in C_{loc}([0,\infty);\E_{loc})$
is a Shatah-Struwe (SS)
 solution of problem \eqref{eq.main} if  $\xi_u\big|_{t=0}=\xi_0$,
\begin{multline}
-\int_0^T(\Dt u,\Dt \phi)dt+\gamma\int_0^T(\Dt u,\phi)dt+\int_0^T(\nabla u,\nabla \phi)dt\,+\\
\int_0^T(u,\phi)dt+\int_0^T(f(u),\phi)dt=\int_0^T(g,\phi)dt,
\end{multline}
for all test functions $\varphi\in C^\infty_0((0,\infty)\times\R^3)$ and, in addition, the
following extra space-time regularity holds:
\begin{equation}\label{4.str}
u\in L^4_{loc}([0,\infty),L^{12}_{loc}(\R^3)).
\end{equation}
\end{Def}

\begin{rem} \label{Rem4.flin} As in the case of finite-energy solutions, extra regularity
 \eqref{4.str} is crucial for the uniqueness of the solution $u$. To the best of our knowledge the
  uniqueness of energy solutions is not known without this assumption even in the finite-energy case.
   Moreover, this assumption is also used in order to derive finite speed propagation inequalities
    which are crucial for the existence result as well.
    \par
    We mention  that, due to estimate \eqref{3.5int} and growth restriction on $f$ this extra
    regularity gives us also that
    \begin{equation}
     f(u)\in L^1_{loc}([0,\infty),L^2_{loc}(\R^3))
    \end{equation}
and, therefore, we may treat the non-linearity $f(u)$ as an external force and use estimate
 \eqref{3.cone} for the obtained linear equation.
\end{rem}

The last remark allows us to verify the uniqueness of SS-solutions.

\begin{prop}\label{Prop4.cone} Let the function $f$ satisfy \eqref{4.f}. Then, for every two SS-solutions $u_1$ and $u_2$
of equation \eqref{eq.main} (which correspond to different initial data and external forces) and
 every $R>0$, $x_0\in\R^3$,
the following analogue of estimate \eqref{3.cone} holds:
\begin{multline}\label{4.cone}
\|\xi_{u_1}(t)-\xi_{u_2}(t)\|_{\E(B^{R-t}_{x_0})}\le\\\le
 C\|\xi_{u_1}(0)-\xi_{u_2}(0)\|_{\E(B^R_{x_0})}+C\int_0^t\|g_1(s)-g_2(s)\|_{L^2(B^{R-s}_{x_0})}\,ds,
\end{multline}
where $0<t<R$ and  the constant $C$ depends on $R$, $x_0$ and the proper Strichartz norms of $u_1$ and $u_2$. In particular,
SS-solution of \eqref{eq.main} is unique.
\end{prop}
\begin{proof} Let $v(t)=u_1(t)-u_2(t)$. Then this function solves the equation
\begin{equation}\label{4.dif}
\Dt^2 v+\gamma\Dt v+(1-\Dx) v=g_1(t)-g_2(t)-[f(u_1(t))-f(u_2(t))].
\end{equation}
Equation \eqref{4.dif} has the form of \eqref{lin.eq} with the right-hand side belonging to
 $L^1_{loc}(\R_+,L^2_{loc}(\R^3))$, therefore, estimate \eqref{3.cone} is applicable and gives
 \begin{multline}\label{4.conen}
\|\xi_v(t)\|_{\E(B^{R-t}_{x_0})}\le C\|\xi_v(0)\|_{\E(B^R_{x_0})}+
C\int_0^t\|g_1(s)-g_2(s)\|_{L^2(B^{R-s}_{x_0})}\,ds+\\+C\int_0^t\|f(u_1(s))-f(u_2(s))\|_{L^2(B^{R-s}_{x_0})}\,ds.
 \end{multline}
 Using assumptions \eqref{4.f} together with H\"older inequality and Sobolev
  embedding $H^1\subset L^6$, we get
 \begin{multline}
\|f(u_1(s))-f(u_2(s))\|_{L^2(B^{R-s}_{x_0})}\le\\\le
C\|(1+|u_1(s)|^4+|u_2(s)|^4)v(s)\|_{L^2(B^{R-s}_{x_0})}\le\\\le
 C(1+\|u_1(s)\|^{4}_{L^{12}(B^{R-s}_{x_0})}+\|u_2(s)\|^4_{L^{12}(B^{R-s}_{x_0})})
 \|v\|_{H^1(B^{R-s}_{x_0})}\le\\\le l(s)\|\xi_v(s)\|_{\E(B^{R-s}_{x_0})},
 \end{multline}
 where
 $$
 l(s):=C(1+\|u_1(s)\|^{4}_{L^{12}(B^{R-s}_{x_0})}+\|u_2(s)\|^4_{L^{12}(B^{R-s}_{x_0})})\in L^1(0,R),
$$
due to the extra regularity assumption \eqref{4.str}. Inserting the obtained estimate
 into the right-hand side of \eqref{4.conen} and applying the Gronwall inequality, we end up
  with the desired estimate \eqref{4.cone} and finish the proof of the proposition.
\end{proof}
As in the linear case, estimate \eqref{4.cone} allows to reduce the study of a general infinite energy
 case to the case of finite-energy solutions where the global well-posedness is known. Namely, we need
  the following result for the finite-energy case which is proved in \cite{SZ-UMN}.

\begin{prop}\label{Prop4.fin} Let the function $f$ satisfy \eqref{4.f},
$\xi_0\in \E$ and the external force $g\in L^1_{loc}(\R_+,L^2(\R^3))$.
Then problem \eqref{eq.main} possesses a unique global SS solution $u(t)$ and the following estimate holds:
\begin{equation}\label{4.finen}
\|\xi_u(t)\|_{\E}+\|v\|_{L^4(t,t+1;L^{12})}\le Q(\|\xi_0\|_\E)e^{-\alpha t}+Q(\|g\|_{L^1_b(0,t+1;L^2)}),
\end{equation}
where the positive constant $\alpha$ and monotone increasing function $Q$
 are independent of $\xi_0$, $t\ge0$ and $g$.
\end{prop}
Combining Propositions \ref{Prop4.fin} and \ref{Prop4.cone}, we get the following result.

\begin{theorem}\label{Th4.wp} Let $\xi_0\in\E_{loc}$ and $g\in L^1_{loc}(\R_+,L^2_{loc}(\R^3))$ and
 let the nonlinearity $f$ satisfy \eqref{4.f}. Then problem \eqref{eq.main} possess a unique
  globally defined SS-solution $u(t)$. Moreover, this solution satisfies the following estimate:
\begin{multline}\label{4.infen}
\|\xi_u(t)\|_{\E(B^R_{x_0})}+\|u(t)\|_{L^4(t,t+1;L^{12}(B^R_{x_0}))}\le\\\le
 Q(\|\xi_0\|_{\E(B^{R+t+1}_{x_0})})e^{-\alpha t}+ Q(\|g\|_{L^1_b(0,t+1,L^2(B^{R+t+1}_{x_0}))}),
\end{multline}
where the constant $\alpha>0$ and monotone increasing function $Q$ are independent
 of $R>0$, $x_0\in\R^3$, $\xi_0$, $g$ and $t>0$.
\end{theorem}
\begin{proof} Indeed, the uniqueness of a solution is verified in Proposition \ref{Prop4.cone}. To construct
the desired solution $u$ we utilize this proposition again. Namely, to get the value of $u$
in a cone $C^R_{0}$ for a given $R>0$, we construct the initial data $\tilde\xi_0\in\E$
  using the extension operator from $\E(B^{R}_0)$ to $\E$, take $\tilde g$ as zero extension of $g$ from
   the cone $C^{R}_0$ to the whole space and solve equation \eqref{eq.main} with the  data $\tilde\xi_0$ and
    $\tilde g$. Let $\tilde u$ be the corresponding finite energy SS solution which exists due to Proposition
     \ref{Prop4.fin}.
\par
Then, $u=\tilde u\big|_{C^{R}_0}$ is a desired SS solution of the initial
problem \eqref{eq.main} in the cone $C^{R}_0$. Moreover, due to Proposition \ref{Prop4.cone},
 this definition is independent of the choice of $R$  (the solutions defined using
  different cones will coincide on a smaller cone). Therefore, increasing $R$, we get the
   required global SS-solution $u(t,x)$ of problem \eqref{eq.main}. Thus, the existence of a
    solution is also verified. Estimate \eqref{4.infen} is also an immediate corollary of
     \eqref{4.finen}, \eqref{4.cone} and the cut-off procedure described above and the theorem is proved.
\end{proof}

We turn now to study the dissipativity of equation \eqref{eq.main}. We first note that even in the
 linear case this problem is {\it not dissipative} if we consider the initial data with sufficiently rapid
  growth rate as $|x|\to\infty$ (this can be easily seen using the explicit formula for
   solutions in the linear case), so at least some restrictions on this growth rate should be posed in
    order to avoid growing in time solutions. Following the standard approach (see \cite{MirZel} and references therein for more details),
    we will consider problem \eqref{eq.main} in the properly chosen uniformly local spaces. Namely,
    we assume from now on that
    \begin{equation}\label{4.ul}
    \xi_0\in\E_b:=H^1_b(\R^3)\times L^2_b(\R^3),\ \ g\in L^1_b(\R_+,L^2_b(\R^3))
    \end{equation}
 and study problem \eqref{eq.main} in the uniformly local energy phase space $\E_b$. The following theorem
  can be considered as the main result of this section.

\begin{theorem}\label{Th4.dis}
	Let the assumptions of Theorem \ref{Th4.wp} holds and let, in addition, \eqref{4.ul} be satisfied.
 Then the SS-solution $u(t)$ of problem \eqref{eq.main} belongs to $\E_b$ for all $t\ge0$
  and satisfies the following estimate:
	\begin{equation}\label{4.dis}
	\|\xi_u(t)\|^2_{\E_b}+\|u\|_{L^4(t,t+1;L^{12}_b)}\le
 Q(\|\xi_0\|_{\E_b})e^{-\beta t}+Q(\|g\|_{L^1_b(\R_+,L^2_b)}),
	\end{equation}
	for some constant $\beta>0$ and monotone nondecreasing function $Q$ which are independent of $u$ and $t\ge0$.
\end{theorem}
\begin{proof} Estimate \eqref{4.dis} can be deduced also from the basic estimate \eqref{4.infen}, but then we need
 the explicit form of the function $Q$ there, so we prefer to argue in a slightly different way. Namely,
 we first note that the regularity $\xi_u(t)\in\E_b$ follows immediately from \eqref{4.infen}. Moreover, the following
  energy-to-Strichartz estimate is also guaranteed by this estimate:
  \begin{equation}\label{4.es}
\|u\|_{L^4(t,t+1;L^{12}_b)}\le Q(\|\xi_u(t)\|_{\E_b})+Q(\|g\|_{L^1(t,t+1;L^2_b)}).
  \end{equation}
  Thus, it is enough to verify the dissipative estimate \eqref{4.dis} for
   the energy norm $\|\xi_u(t)\|_{\E_b}$ only. This can be done in a standard way using the
    weighted energy estimates. Namely, we need to multiply equation \eqref{eq.main} by $\phi_{\eb,x_0}^2(\Dt u+\delta u)$
    for some small positive $\delta$, get the weighted analogue of the standard energy estimate and
     finally take a supremum over $x_0\in\R^3$ to derive the desired uniformly local energy estimate
     \begin{equation}\label{4.en-u}
\|\xi_u(t)\|^2_{\E_b}\le Q(\|\xi_u(0)\|_{\E_b})e^{-\alpha t}+Q(\|g\|_{L^1_b(\R_+,L^2_b)}),
     \end{equation}
see \cite{MS,Zhyp} for the details. Since these arguments will be repeated in more details in the next
 section, we omit these details here. Combining \eqref{4.en-u} and \eqref{4.es} we get the desired
  dissipative estimate and finish the proof of the theorem.
\end{proof}

\section{Asymptotic smoothing property}\label{s5}

In this section we verify that any SS-solution of our problem \eqref{eq.main} can be split to
 exponentially decaying and more regular  parts. For simplicity, we will consider only
  the case of autonomous equation, so we assume from now on that
  \begin{equation}\label{5.g}
g\in L^2_b(\R^3).
  \end{equation}
The general case, say $g\in L^1_b(\R_+,H^1_b(\R^3))$ can be treated analogously, but we need not
this since the study of non-autonomous attractors is out of scope of this paper.
\par
Following \cite{BV,SZ-UMN,ZCPAA2004}, we split the solution $u(t)$ of problem \eqref{eq.main} as follows:
\begin{equation}\label{5.vw}
u(t)=v(t)+w(t).
\end{equation}
The decaying component $v(t)$ is chosen to satisfy the following equation:
\begin{equation}\label{5.eq.v}
\Dt^2 v+\gamma\Dt v+(1-\Dx)v+Lv+f(v)=0,\ \ \xi_v\big|_{t=0}=\xi_u(0),
\end{equation}
where $L$ is a sufficiently big positive number which will be fixed below. Finally, the
 smooth reminder $w(t)$ solves the equation
\begin{equation}\label{5.eq.w}
\Dt^2 w+\gamma\Dt w+(1-\Dx)w+f(u)-f(v)=Lv+g,\ \xi_{w}\big|_{t=0}=0.
\end{equation}
We start with the $v$-component.
\begin{prop}\label{Lem5.v}Let the assumptions of Theorem \ref{Th4.dis} holds and let,
in addition, $\xi_u(0)\in\E_b$ and $g$ enjoys \eqref{5.g}. Then, for sufficiently large
 $L=L(f)$, the SS-solution $v(t)$ of
 problem \eqref{5.eq.v} satisfies the following estimate:
\begin{equation}\label{5.vdec}
\|\xi_{v}(t)\|_{\E_b}+\|v\|_{L^4(t,t+1;L^{12}_b)}\le Q(\|\xi_u(0)\|_{\E_b})e^{-\beta t},
\end{equation}
where the constant $\beta>0$ and non-decreasing function $Q$   are independent of $\xi_u(0)\in\E_b$,
 $t\ge0$ and $g\in L^2_b$.
\end{prop}
\begin{proof} Let us fix sufficiently small $\eb>0$ and $x_0\in\R^3$. Then, multiplying
 equation \eqref{5.eq.v} by $\phi^2_{\eb,x_0}\Dt v+\kappa \phi^2_{\eb,x_0}v$
 (with small $\kappa>0$ which will
 be fixed later) and integrating over $\R^3$, we find
\begin{equation}\label{5.v-enw}
	\frac{d}{dt}E_v(t)+\kappa E_v(t)+P_v(t)=0,
\end{equation}
where
\begin{multline}
E_v(t):=\frac12\|\xi_v(t)\|^2_{\E_{\phi_{\eb,x_0}}}+
\(\phi^2_{\eb,x_0},F(v(t))+\frac L2v^2(t)\)+\\+
\k\(\phi_{\eb,x_0}^2v,\Dt v\)+
\frac{\k\gamma}{2}\|v(t)\|^2_{L^2_{\phi_{\eb,x_0}}}
\end{multline}		
and
\begin{multline}
P_v(t):=\frac12\((2\gamma-3\k)\|\Dt v(t)\|^2_{L^2_{\phi_{\eb,x_0}}}+
		\k\|\nabla v(t)\|^2_{L^2_{\phi_{\eb,x_0}}}\)+\\+
\frac12\kappa(1-\k\gamma)\|v\|^2_{L^2_{\phi_{\eb,x_0}}}-\k^2(\phi_{\eb,x_0} v,\phi_{\eb,x_0}\Dt v(t))_{L^2}+\\+
\kappa\(\phi^2_{\eb,x_0},f(v(t))v(t)+\frac L2v^2(t)-F(v(t))\)+\\+
2(\phi_{\eb,x_0}\nabla v(t),\nabla \phi_{\eb,x_0}\Dt v(t))+
2\k(\phi_{\eb,x_0}\nabla v(t),\nabla \phi_{\eb,x_0}v(t)).
\end{multline}
Note that our assumptions \eqref{4.f} on the nonlinearity $f$ give the following inequalities:
$$
-Kv^2 \le F(v)\le f(v)v+Kv^2
$$
for some positive $K$. By this reason, all terms in the definitions of $E_v(t)$ and $P_v(t)$ which contain
 the nonlinearity will be non-negative if we take $L\ge 2K$.
\par
Fixing now $\kappa>0$ and $\eb>0$ small enough and using   inequality \eqref{2.dw}, we conclude that
 $$
 P_v(t)\ge0
 $$
 and
$$
 \frac14\|\xi_v(t)\|_{\E_{\phi_{\eb,x_0}}}^2\le
  E_v(t)\le C(\|\xi_v(t)\|_{\phi_{\eb,x_0}}^2+(\phi_{\eb,x_0}^2,F(v)),
$$
we deduce from \eqref{5.v-enw} that
$$
\frac d{dt} E_v(t)+\kappa E_v(t)\le0.
$$
The Gronwall inequality together with the fact that $f(v)$ has a quintic
 growth rate and the embedding $H^1\subset L^6$
now give
$$
\|\xi_v(t)\|_{\E_{\phi_{\eb,x_0}}}^2\le C\(\|\xi_v(0)\|^2_{\E_{\phi_{\eb,x_0}}}+
C\|\xi_v(t)\|^6_{\E_{\phi_{\eb/3,x_0}}}\)e^{-\beta t},
$$
for some positive constants $C$ and $\beta$ which are independent of $x_0$. Taking the
 supremum over $x_0\in\R^3$, we arrive at the desired dissipative estimate \eqref{5.vdec} for the
  energy norm $\|\xi_v(t)\|^2_{\E_b}$. So, it only remains to obtain its analogue for the
   Strichartz norm. We will do it in two steps.
   \par
   {\it Step 1.} We apply energy-to-Strichartz estimate \eqref{4.es} (where $g=0$ and $f$ is
    replaced by $f_L(v)=f(v)+Lv$) to get
    \begin{equation}\label{5.bad}
\|v\|_{L^4(t,t+1;L^{12}_b)}\le Q(\|\xi_v(t)\|_{\E_b})\le Q(\|\xi_v(0)\|_{\E_b}).
    \end{equation}
Here we are unable to get the decaying estimate since we do not know that $Q(0)=0$
 (it is likely so, but to check this we need to revise the proof given in \cite{SZ-UMN} as
  well as the proof of energy-to-Strichartz estimates given in \cite{BG1999} or \cite{tao1}
     which we prefer not to do). So we need one more step.
   \par
   {\it Step 2.} We apply {\it linear} Strichartz estimate to  equation \eqref{5.eq.v}
   treating the term $f_L(v)$ as a perturbation to get:
\begin{equation}
\|v\|_{L^4(t,t+1;L^{12}_b)}\le C(\|\xi_v(t)\|_{\E_b}+C\|f_L(v)\|_{L^1(t,t+1;L^2_b)})
\end{equation}
Using estimate \eqref{3.5int} and our assumptions \eqref{4.f} on the non-linearity, we deduce
\begin{equation}
\|f_L(v)\|_{L^1(t,t+1;L^2_b)}\le C\|\xi_v(t)\|_{L^\infty(t,t+1;\E_b)}\(1+\|v\|^4_{L^4_b(t,t+1;L^{12}_b)}\).
\end{equation}
Combining these estimates with the already proved decaying estimate for the energy norm, we get
 the desired estimate \eqref{5.vdec} and finish the proof of the proposition.	
\end{proof}
We now turn to the most complicated $w$-component. Note first of all that estimates \eqref{4.dis}
 and \eqref{5.vdec} give
 \begin{equation}\label{5.wen}
\|\xi_w(t)\|_{\E_b}+\|w\|_{L^4(t,t+1;L^{12}_b)}\le
 Q(\|\xi_u(0)\|_{\E_b})e^{-\beta t}+Q(\|g\|_{L^2_b}),
 \end{equation}
 but we need an analogue of this estimate in $\E^\alpha_b$ for some $\alpha>0$. We will derive it in two steps.
 At the first step we derive the exponentially divergent analogue of this higher energy estimate which will be
  improved at the next step.

\begin{prop}\label{Prop5.wdiv}
	Let the above assumptions hold and let $\alpha\in(0,2/5]$. Then the $w$-component of the SS-solution
 $u$ of problem \eqref{eq.main} satisfies the following estimate:
\begin{equation}\label{5.wgr}
	\|\xi_w(t)\|_{\E^\alpha_b}+\|w\|_{L^4([t,t+1];H^{\alpha,12}_b)}\le
 e^{Kt}\(Q(\|\xi_u(0)\|_{\E_b})+Q(\|g\|_{L^2_b})\),
	\end{equation}
	for some monotone function  $Q$, which
 does not depend on $\xi_u(0)$ and $g$.	
\end{prop}
\begin{proof} Let the cut-off function $\psi(x)\in C^\infty_0(\R^3)$
 be the same as in Remark
 \ref{Rem2.cut} and let
 $$
 \psi_{R,x_0}(x):=\psi_0\(\frac{x-x_0}R\),
 $$
where the parameter $R\ge1$ will be specified below (for the proof of this proposition
 we may fix $R=1$, for what follows later we need $R\gg1$). Then, we have obvious estimates
 \begin{equation}\label{5.cut}
|\Nx\psi_{R,x_0}(x)|+|\Dx\psi_{R,x_0}(x)|\le CR^{-1},
 \end{equation}
 where $C$ is independent of $R$ and $x_0$.
Let us set
 $$
 v_{x_0}=\psi_{R,x_0}v, \ \ w_{x_0}=\psi_{R,x_0}w\ \ \text{and}\ \  u_{x_0}:=\psi_{R,x_0}u.
 $$
  Then $w_{x_0}$ solves
\begin{multline}\label{eq.wx0}
\Dt^2 w_{x_0}+\gamma\Dt w_{x_0}-\Dx w_{x_0}+w_{x_0}=-\psi_{R,x_0}(f(u)-f(v))-\\-
\Dx\psi_{R,x_0}w-2\nabla\psi_{R,x_0}\nabla w+ Lv_{x_0}+\psi_{R,x_0}g,\ \
\xi_{w_{x_0}}|_{t=0}=0.
\end{multline}
Also, without loss of generality, we may assume that $f'(0)=0$. Indeed, in general case we may just replace
$f(u)$ by $\tilde f(u):=f(u)-f'(0)u$ and the extra term $f'(0)(u-v)$ which will appear in the
 right-hand side of equation \eqref{5.eq.w} is under the control due to estimates \eqref{4.dis}
  and \eqref{5.wen}.
 \par
The right-hand side of this equation contains the term $\psi_{R,x_0}g$ which is only
 $L^2$ and not $H^\alpha$ and this prevents us to use the linear Strichartz
  estimate \eqref{lin.est} directly. To overcome this difficulty, we introduce the functions
   $$
   \theta_{x_0}:=(1-\Dx)^{-1}(\psi_{R,x_0}g) \ \ \text{and}\ \ \ \bar w_{x_0}:=w_{x_0}-\theta_{x_0}.
   $$
Then the last function solves
\begin{multline}\label{eq.wx01}
\Dt^2 \bar w_{x_0}+\gamma\Dt \bar w_{x_0}+(1-\Dx) \bar w_{x_0}=-\psi_{R,x_0}(f(u)-f(v))-\\-
\Dx\psi_{R,x_0}w-2\nabla\psi_{R,x_0}\nabla w+ Lv_{x_0},\ \
\xi_{w_{x_0}}\big|_{t=0}=\{-\theta_{x_0},0\}
\end{multline}
and we may apply \eqref{lin.est} to this equation instead. Note also that due to the
elliptic regularity and Sobolev embedding theorem,
$$
\|\theta_{x_0}\|_{H^{2}}+\|\theta_{x_0}\|_{H^{\alpha,12}}\le C\|g\|_{L^2_b}
$$
(here we have used that $\alpha\le2/5$),
 there is no difference in estimation the corresponding energy and Strichartz norms of $\bar w_{x_0}$ and $w_{x_0}$.
\par
 Applying linear Strichartz estimates \eqref{lin.est} to equation \eqref{eq.wx01}
and using \eqref{4.dis} and \eqref{5.vdec} together with \eqref{5.cut},  we find
\begin{multline}
\label{sm.est3}
\|\xi_{w_{x_0}}(t)\|_{\E^\alpha}+
\left(\int_0^t e^{-4\beta(t-s)}\|w_{x_0}(s)\|^4_{H^{\alpha,12}}\,ds\right)^\frac{1}{4}\le
\\\le
C\int_0^te^{-\beta(t-s)}\|\psi_{x_0}(f(u)-f(v))\|_{H^\alpha}\,ds+\\+
CR^{-1}\int_0^te^{-\beta(t-s)}\|\xi_w(s)\|_{\E^\alpha_{b,2R}}\,ds+Q_R(\|\xi_u(0)\|_{\E_b})e^{-\beta t}
+Q_R(\|g\|_{L^2_b}),
\end{multline}
where the constants $C$ and $\beta>0$ are independent of $R$ and $x_0$, the monotone function $Q_R$
 may depend on $R$ (but not on $x_0$) and
 $$
 \|\xi_v\|_{\E_{b,R}^\alpha}:=\sup_{x_0\in\R^3}\|\xi_v\|_{\E^\alpha(B^{R}_{x_0})}.
 $$
 The key problem is to estimate the integral in the right-hand side of  \eqref{sm.est3}
 which contains non-linearity $f$. To this end, we use estimate
 \eqref{A.mf}, see Appendix A, together with \eqref{4.dis} and \eqref{5.vdec} to derive
\begin{multline}\label{sm.est3'}
\|\psi_{x_0}(f(u)-f(v))\|_{H^\alpha(\R^3)}\le\\
 C\left(1+\|u\|_{L^{12}(B^{2R}_{x_0})}+\|v\|_{L^{12}(B^{2R}_{x_0})}\right)^{4-\alpha}
 \(1+\|u\|_{H^1(B^{2R}_{x_0})}+\|v\|_{H^1(B^{2R}_{x_0})}\)^\alpha\times\\
 \|w_{x_0}\|^{1-\alpha}_{H^{1+\alpha}(\R^3)}\|w_{x_0}\|^{\alpha}_{H^{\alpha,12}(\R^3)}
 \le m_{R,x_0}(t)^{1-\frac\alpha4}\|w_{x_0}\|^{1-\alpha}_{H^{1+\alpha}(\R^3)}\|w_{x_0}\|^{\alpha}_{H^{\alpha,12}(\R^3)},
\end{multline}
where
\begin{equation}\label{5.m}
m_{R,x_0}(t):=K_R\left(1+\|u(t)\|^4_{L^{12}(B^{2R}_{x_0})}+\|v(t)\|^4_{L^{12}(B^{2R}_{x_0})}\right)
\end{equation}
and the constant $K_R=K_R(\|\xi_u(0)\|_{\E_b},\|g\|_{L^2_b})$.
\par
Using now  Holder's inequality in time with
exponents $\tfrac{4}{4-\alpha}$ and $\tfrac{4}{\alpha}$ we have the chain of inequalities as follows
\begin{multline}\label{sm.est3''}
 \int_0^te^{-\beta(t-s)}\|\psi_{R,x_0}(f(u)-f(v))\|_{H^\alpha}\,ds\le\\\le
\left(\int_0^te^{-k_\alpha\beta(t-s)}m_{R,x_0}(s)
\|w_{x_0}\|^{k_\alpha}_{H^{1+\alpha}}\,ds\right)^{1-\alpha/4}\times\\
\left(\int_0^te^{-4\beta(t-s)}\|w_{x_0}\|^4_{H^{\alpha,12}}\,ds\right)^\frac{\alpha}{4}\le
C\left(\int_0^te^{k_\alpha\beta(t-s)}m_{R,x_0}\|w_{x_0}\|^{k_\alpha}_{H^{1+\alpha}}\,ds\right)^{\frac1{k_\alpha}}\\+
\frac{1}{4}\left(\int_0^te^{-4\beta(t-s)}\|w_{x_0}\|^4_{H^{\alpha,12}}\,ds\right)^\frac{1}{4},
\end{multline}
where $k_\alpha:=\frac{4-4\alpha}{4-\alpha}$ and at the last step we used Young's
inequality with exponents $\tfrac{1}{\alpha}$ and $\tfrac{1}{1-\alpha}$.
\par
The second integral will be cancelled with the left-hand side of \eqref{sm.est3} and for the first one
we continue the estimate using Young's inequality
 in time with exponents $\tfrac{1}{k_\alpha}$ and $\tfrac{1}{1-k_\alpha}$:
\begin{multline}\label{sm.est5'}
\left(\int_0^te^{-k_\alpha\beta(t-s)}m_{R,x_0}(s)
\|\xi_{w_{x_0}}(s)\|^{k_\alpha}_{\E^\alpha}\,ds\right)^\frac{1}{k_\alpha}=\\
\left(\int_0^te^{-k^2_\alpha\beta(t-s)}m^{k_\alpha}_{R,x_0}(s)
\|\xi_{w_{x_0}}(s)\|^{k_\alpha}_{\E^\alpha}e^{-k_\alpha(1-k_\alpha)\beta(t-s)}
m^{1-k_\alpha}_{R,x_0}(s)\,ds\right)^\frac{1}{k_\alpha}\leq\\
\left(\int_0^te^{-k_\alpha\beta(t-s)}m_{R,x_0}(s)
\|\xi_{w_{x_0}}(s)\|_{\E^\alpha}\,ds\right)\left(\int_0^t
e^{-k_\alpha\beta(t-s)}m_{R,x_0}(s)\,ds\right)^\frac{1-k_\alpha}{k_\alpha}\leq\\\le
K_R\int_0^te^{-k_\alpha\beta(t-s)}m_{R,x_0}(s)\|\xi_{w_{x_0}}(s)\|_{\E^\alpha}\,ds,
\end{multline}
for some constant $K_R$ depending on $R$, $\|\xi_u(0)\|_{\E_b}$ and $\|g\|_{L^2_b}$ (here we
have also implicitly used estimates \eqref{5.vdec} and \eqref{4.dis}).
Inserting these estimates into the right-hand side of \eqref{sm.est3}, we arrive at
\begin{multline}\label{sm.est5.2}
\|\xi_{w_{x_0}}(t)\|_{\E^\alpha}+\left(\int_0^t e^{-4\beta(t-s)}
\|w_{x_0}(s)\|^4_{H^{\alpha,12}}\,ds\right)^\frac{1}{4}\le\\\le
 K_R\int_0^te^{-k_\alpha\beta(t-s)}m_{R,x_0}(s)\|\xi_{w_{x_0}}(s)\|_{\E^\alpha}\,ds+\\+
 CR^{-1}\int_0^te^{-\beta(t-s)}\|\xi_w(s)\|_{\E^\alpha_{b,2R}}\,ds+Q_R(\|\xi_u(0)\|_{\E_b})e^{-\beta t}
+Q_R(\|g\|_{L^2_b}).
\end{multline}
To complete the proof we need the following version of the Gronwall lemma.

\begin{lemma}\label{Lem8.gr} Let the function $Y\in C_{loc}([\tau,\infty))$ satisfies
	$$
	Y(t)\le H(t)+\int_\tau^te^{-\beta_0(t-s)}(l(s)Y(s)+G(s))\,ds,\ \ \ t\ge\tau
	$$
	for some constant  $\beta_0$, some functions $H\in L^\infty_{loc}([\tau,\infty))$,
$G\in L^1_{loc}([\tau,\infty))$ and non-negative function $l(t)\ge0$ such that $l\in L^1_{loc}([\tau,\infty))$. Then, the following estimate holds:
\begin{equation}\label{8.gr1}
	Y(t)\le H(t)+\int_\tau^t e^{-\beta_0(t-s)+\int_s^tl(\kappa)\,d\kappa}(l(s)H(s)+G(s))\,ds
	\end{equation}
for all $t\ge\tau$.
\end{lemma}
The proof of this  lemma is standard and is left to the reader.
\par
Using estimates \eqref{4.dis} and \eqref{5.vdec}, we see that
\begin{equation}\label{5.mbig}
K_R\|m_{R,x_0}\|_{L^1_b}\le K,
\end{equation}
where the constant $K$ depends on $R$ and norms $\|\xi_u(0)\|_{\E_b}$ and $\|g\|_{L^2_b}$ (but is
 independent of $x_0$). Applying the Gronwall inequality \eqref{8.gr1} with
\begin{multline*}
 Y(t)=\|\xi_{w_{x_0}}(t)\|_{\E^\alpha},\  l(t)=K_Rm_{R,x_0}(t), \ \beta_0=k_\alpha \beta,\\
 G(t)=\|\xi_w(t)\|_{\E^\alpha_{b,2R}} \text{ and } H(t)=Q_R(\|\xi_u(0)\|_{\E_b})e^{-\beta t}
+Q_R(\|g\|_{L^2_b})
\end{multline*}
to \eqref{sm.est5.2} after
 the straightforward estimations, we arrive at
 \begin{equation}\label{est.w.div}
\|\xi_{w_{x_0}}(t)\|_{\E^\alpha}\le
CR^{-1}\int_0^te^{K(t-s)}\|\xi_w(s)\|_{\E^\alpha_{b,2R}}\,ds+Qe^{Kt},
 \end{equation}
 where the constants $Q$ and $K$ depend on $R$, $\|\xi_u(0)\|_{\E_b}$ and $\|g\|_{L^2_b}$.
 Note that in this estimate  the constant $C$ depends on $K$ since we have used an obvious estimate
 $$
\int_s^tl(\kappa)\,d\kappa\le K+K(t-s),
 $$
 where we cannot avoid the first term $K$ in the right-hand side. Thus, the constant $C$ depends
  on $R$ as well, but only through the constant $K$. In the sequel we modify this estimate
   in such a way that $K$ will be small and independent of $R$, then the constant $C$ will
    automatically be independent of $R$ as well (this observation is not important for the proof of the current
     proposition since we can just fix $R=1$ here, but will be crucial for what follows later).
\par
Taking the supremum with respect to $x_0\in\R^3$ from both parts and using the
  standard inequalities
  $$
  C_1\|\xi_w\|_{\E_{b,R}}\le\|\xi_w\|_{\E_{b,2R}}\le C_2\|\xi_w\|_{\E_{b,R}},
  $$
  where the constants $C_i$ are independent of $R$,
we conclude that
 \begin{multline}\label{5.eb-div}
\|\xi_{w}(t)\|_{\E^\alpha_{b,R}}\le
C'R^{-1}\int_0^te^{K(t-s)}\|\xi_w(s)\|_{\E^\alpha_{b,R}}\,ds+\\+Q_R(\|\xi_u(0)\|_{\E_b}+\|g\|_{L^2_b})e^{Kt},
 \end{multline}
Applying the Gronwall inequality again, we end up with the desired inequality \eqref{5.wgr}
 for the $\E^\alpha_b$-energy. To get the estimate for the Strichartz part it is now enough
 to use \eqref{sm.est5.2}. Thus, the proposition is proved.
\end{proof}
We now want to improve estimate \eqref{5.wgr} and get its dissipative analogue. To this end,
we need to obtain the analogue of estimate \eqref{sm.est5.2}, where the $L^1_b$-norm of $m_{R,x_0}(t)$
will be {\it small}. Fixing also $R$ large enough, the Gronwall inequality would give us the
 desired dissipative estimate. The key idea is to split the solution $u$ in a sum
 \begin{equation}\label{5.split}
 u(t)=\tilde v(t)+\tilde w(t)
 \end{equation}
of more regular ($\widetilde w$) and small ($\widetilde v$) parts using already
proved propositions \ref{Lem5.v} and \ref{Prop5.wdiv} and then replace the function $u$ in \eqref{5.m} by
its small part $\tilde v$ (estimating the term $f(u)-f(v)$ in a more accurate way). To this end,
we need the following lemma.

\begin{lemma}\label{Lem8.usplit} Let the assumptions  of Proposition \ref{Lem5.v}. Then, for every
 $\delta>0$ and $\alpha\in(0,\frac25]$ there exists time $T_\delta$ depending on $\|\xi_u(0)\|_{\E_b}$ such that the SS-solution
  $u(t)$ of problem \eqref{eq.main} possesses
   decomposition \eqref{5.split} such that, for all $t\ge T_\delta$,
	\begin{equation}\label{8.small1}
	 \|\xi_{\tilde v}(t)\|_{\E_b}+\|\tilde v\|_{L^4(t,t+1;L^{12}_b)}\le \delta
	\end{equation}
and
	\begin{equation}\label{8.smooth}
	 \|\xi_{\tilde w}(t)\|_{\E_b^\alpha}+\|\tilde w\|_{L^4(t,t+1;H^{\alpha,12}_b)}\le M_\delta,
	\end{equation}
 where the constant $M_\delta$ depends on $\|g\|_{L^2_b}$ and $\delta$, but is
  independent of $t$ and the norm $\|\xi_u(0)\|_{\E_b}$ of the initial data.	
\end{lemma}
\begin{proof} We first note that, due to the dissipative estimate \eqref{4.dis}, the solution $u(t)$
  satisfies the estimate
 \begin{equation}\label{5.abs}
 \|\xi_u(t)\|_{\E_b}+\|u\|_{L^4(t,t+1;L^{12}_b)}\le M_0:=2Q(\|g\|_{L^2_b})
\end{equation}
 for all $t\ge T'$, where $T'$ depends on the norm of the initial data only. By this reason,
 we may assume without loss of generality that $T'=0$ and the solution $u$ satisfies \eqref{5.abs}
  from the very beginning.
\par
   Let us now fix a big number $T=T(\delta)$ and
  consider decompositions $u(t)=v_n(t)+w_n(t)$, $t\ge nT$ which are defined by equations \eqref{5.eq.v}
   and \eqref{5.eq.w}, but starting with the time moment $T_n=T(n-1)$ with the initial data
   $$
    \xi_{v_n}\big|_{t=T_n}=\xi_u\big|_{t=T_n},\ \ \xi_{w_n}\big|_{t=T_n}=0.
   $$
   Then, due to estimates \eqref{5.vdec} and \eqref{5.wgr} we get
   \begin{equation}\label{5.1n}
   \|\xi_{v_n}(t)\|_{\E_b}+\|v_n\|_{L^4_b(t,t+1;L^{12}_b)}\le Q(M_0)e^{-\beta t}\le\delta
   \end{equation}
   if $t\ge T_n+T=Tn$ and $T$ is chosen to satisfy
    $$
    Q(M_0)e^{-\beta T}=\delta
    $$
    and
   \begin{equation}\label{5.2n}
\|\xi_{w_n}(t)\|_{\E_b}+\|w_n\|_{L^4_b(t,t+1;L^{12}_b)}\le e^{Kt}(Q(M_0)+Q(\|g\|_{L^2_b}))\le M_\delta
   \end{equation}
   if $t\le T_n+2T=T(n+1)$ and $M_\delta:=e^{2KT}(Q(M_0)+Q(\|g\|_{L^2_b}))$.
\par
Finally, we define the desired functions $\tilde v(t)$ and $\tilde w(t)$ for $t\ge T_\delta:=T$ as piece-wise
continuous hybrid functions:
\begin{equation}\label{hybrid}
\tilde v(t):=v_n(t),\ \ \tilde w(t):=w_n(t),\ \ t\in [nT,(n+1)T).
\end{equation}
The desired properties of $\tilde v$ and $\tilde w$ are now guaranteed by estimates \eqref{5.1n} and \eqref{5.2n}
 and the lemma is proved. Crucial for the above construction is that we define the functions $v_n(t)$ and $w_n(t)$
 starting from the initial time $T_n=T(n-1)$, but using these functions in \eqref{hybrid} on the time
  interval $t\in[Tn,T(n+1)]$ only and this time shift of length $T$ guarantees that $v_n(t)$ is already
  small for	$t\in[Tn,T(n+1)]$ due to exponential decay of the $v_n$-component. Thus, the proposition is proved. 	
\end{proof}
Since the function $v(t)$ defined by \eqref{5.eq.v} is exponentially decaying, we may also assume that
 $T_\delta$ is large enough that
 \begin{equation}\label{5.vd}
\|\xi_v(t)\|_{\E_b}+\|v\|_{L^4(t,t+1;L^{12}_b)}\le \delta,\ \ t\ge T_\delta.
 \end{equation}
We are now ready to obtain a non-growing estimate for $w$.
\begin{prop}\label{Prop5.ws}
	Let assumptions of Proposition  \ref{Lem5.v} hold.
Then the solution $w$ of problem \eqref{5.eq.w} obeys the estimate
	\begin{equation}\label{5.w-dis}
		\|\xi_w(t)\|_{\E^\alpha_b}+\|w\|_{L^4([t,t+1];H^{\alpha,12}_b)}\le
 Q(\|\xi_u(0)\|_{\E_b})e^{-\beta t}+Q(\|g\|_{L^2_b}),
	\end{equation}
	for some positive  constant $\beta$ and monotone function $Q$ which do not depend on $t$
 and $\xi_u(0)$.
\end{prop}
\begin{proof} We will utilise again estimate \eqref{sm.est3}, but will estimate the difference $f(u)-f(v)$
 in a more accurate way. To this end, we first note that without loss of generality we may assume
  that $T_\delta=0$ in Lemma \ref{Lem8.usplit}. Indeed, in general case we just apply Strichartz estimate \eqref{lin.est}
  to equation for $w$ starting not from $t=0$, but from $t=T_\delta$ and estimate the initial
  data $\xi_w\big|_{t=T_\delta}$ using Proposition \ref{Prop5.wdiv}. This will give us the same
   type of estimate \eqref{sm.est3} up to maybe different function $Q_R$. As in the proof of
    Proposition \ref{Prop5.wdiv}, we also assume that $f'(0)=0$.
    \par
We split the difference $f(u)-f(v)$ as follows
$$
f(u)-f(v)=[f(\tilde v+\tilde w)-f(\tilde v)]+[f(\tilde v)-f(v)]
$$
The first term can be estimated using inequality \eqref{A.mf} exactly as in the proof
of Proposition \eqref{Prop5.wdiv}. Thus, due to estimates
\eqref{8.small1} and \eqref{8.smooth}, we have
\begin{equation}\label{5.ww}
\int_0^te^{-\beta(t-s)}\|\psi_{R,x_0}(f(\tilde v+\tilde w)-f(\tilde v))\|_{H^\alpha}\,ds\le Q_R=Q_R(M_0,\delta).
\end{equation}
To estimate the second term we note that
$$
\tilde v-v=(u-\tilde w)-(u-w)=w-\tilde w
$$
and, therefore, due to \eqref{A.mfs},
\begin{multline}
\|\psi_{R,x_0}(f(\tilde v)-f(v)\|_{H^\alpha}\le
 C(1+\|\tilde v\|_{L^{12}(B^{2R}_{x_0})}+\| v\|_{L^{12}(B^{2R}_{x_0})})^{4-\alpha}\times\\\times
 (\|\tilde v\|_{H^1(B^{2R}_{x_0})}+\| v\|_{H^1(B^{2R}_{x_0})})^\alpha\times\\\times
 \(\|\psi_{R,x_0}w\|^{1-\alpha}_{H^{1+\alpha}}\|\psi_{R,x_0} w\|^\alpha_{H^{\alpha,12}}+
 \|\psi_{R,x_0}\tilde w\|^{1-\alpha}_{H^{1+\alpha}}\|\psi_{R,x_0} \tilde w\|^\alpha_{H^{\alpha,12}}\).
\end{multline}
The term containing $\tilde w$ can be estimated exactly as in \eqref{5.ww} and using \eqref{8.small1}
 and \eqref{5.vd} for estimating the first term, we arrive at
\begin{multline}\label{5.ww1}
\int_0^te^{-\beta(t-s)}\|\psi_{R,x_0}(f(\tilde v)-f( v))\|_{H^\alpha}\,ds\le\\
\delta^\alpha\int_0^te^{-\beta(t-s)}\tilde m_{R,x_0}(s)^{1-\alpha/4}\|\psi_{R,x_0}w\|^{1-\alpha}_{H^{1+\alpha}}
\|\psi_{R,x_0} w\|^\alpha_{H^{\alpha,12}}\,ds
 +Q_R(M_0,\delta),
\end{multline}
where
$$
\tilde m_{R,x_0}(t):=K_R(1+\|v(t)\|^4_{L^{12}(B^{2R}_{x_0})}+\|\tilde v(t)\|^4_{L^{12}(B^{2R}_{x_0})}).
$$
Inserting the obtained estimates into \eqref{sm.est3} and using the H\"older inequality (exactly as in the
 proof of Proposition \ref{Prop5.wdiv}), we end up with
\begin{multline}\label{sm.est-good}
\|\xi_{w_{x_0}}(t)\|_{\E^\alpha}+\left(\int_0^t e^{-4\beta(t-s)}
\|w_{x_0}(s)\|^4_{H^{\alpha,12}}\,ds\right)^\frac{1}{4}\le\\\le
 \delta^{\frac{7\alpha}{4-4\alpha}}\int_0^te^{-k_\alpha\beta(t-s)}\tilde m_{R,x_0}(s)\|\xi_{w_{x_0}}(s)\|_{\E^\alpha}\,ds+\\+
 CR^{-1}\int_0^te^{-\beta(t-s)}\|\xi_w(s)\|_{\E^\alpha_{b,2R}}\,ds+Q_R(\|\xi_u(0)\|_{\E_b})e^{-\beta t}
+Q_R(\|g\|_{L^2_b}),
\end{multline}
where in contrast to \eqref{sm.est5.2} we have an extra small parameter $\delta$. Crucial for us that
 the $L^1_b$-norm of $\tilde m_{R,x_0}$ is independent of $\delta$. Therefore, for every fixed $R$ we may fix
  $\delta=\delta(R)$ such that
  $$
  \delta^{\frac{7\alpha}{4-4\alpha}}\|\tilde m_{R,x_0}\|_{L^1_b}\le \frac12 k_\alpha\beta
  $$
  and, therefore, the Gronwall inequality \eqref{8.gr1} applied to \eqref{sm.est-good} gives the
   dissipative analogue of \eqref{est.w.div}
 \begin{multline}\label{est.w.ddis}
\|\xi_{w_{x_0}}(t)\|_{\E^\alpha}\le
\frac{C'}R\int_0^te^{-k_\alpha\beta(t-s)/2}\|\xi_w(s)\|_{\E^\alpha_{b,2R}}\,ds+\\+
Q_R(\|\xi_u(0)\|_{\E_b})e^{-k_\alpha\beta t/2}+Q_R(\|g\|_{L^2_b}).
 \end{multline}
It is important that the constant $C'$ here is independent of $R$.
 Taking the supremum over $x_0\in\R^3$ from both sides of this inequality, analogously to \eqref{5.eb-div} we arrive at
 \begin{multline}
\|\xi_{w}(t)\|_{\E^\alpha_{b,R}}\le
\frac {C''}{R}\int_0^te^{-k_\alpha\beta(t-s)/2}\|\xi_w(s)\|_{\E^\alpha_{b,R}}\,ds+\\+
Q_R(\|\xi_u(0)\|_{\E_b})e^{-k_\alpha\beta t/2}+Q_R(\|g\|_{L^2_b}).
 \end{multline}
Fixing now $R>0$ large enough that $C''/R\le\frac14k_\alpha\beta$ and applying the Gronwall inequality again,
 we get the desired dissipative estimate for $\|\xi_w(t)\|_{\E_b}$:
 $$
 \|\xi_w(t)\|_{\E_{b,R}}\le Q(\|\xi_u(0)\|_{\E_{b,R}})e^{-k_\alpha \beta t/4}+Q(\|g\|_{L^2_b}).
 $$
 The dissipative estimate for the Strichartz norm follows now from  \eqref{sm.est-good} exactly as in
  Proposition \ref{Prop5.wdiv}. So, the proposition is proved.
\end{proof}
The next corollary gives the well-posedness and dissipativity of solutions of equation
 \eqref{eq.main} in higher energy spaces.

\begin{cor}\label{Cor5.adis} Let the assumptions of Proposition \ref{Lem5.v} hold and let, in addition,
$\xi_u(0)\in\E_b^\alpha$ for some $\alpha\in(0,\frac25]$.
  Then the corresponding solution $u(t)\in\E_b^\alpha$ for all $t\ge0$ and the following estimate holds:
  \begin{equation}\label{5.aldis}
\|\xi_u(t)\|_{\E_b^\alpha}+\|u\|_{L^4(t,t+1;H^{\alpha,12}_b)}\le
Q(\|\xi_u(0)\|_{\E^\alpha_b})e^{-\beta t}+Q(\|g\|_{L^2_b}),
  \end{equation}
  where the positive constant $\beta$ and monotone increasing function $Q$ may depend on $\alpha$,
  but is independent on $g$, $u$ and $t$.
\end{cor}
Indeed, the proof of this estimate follows word by word to the proof of Proposition \ref{Prop5.ws} and
even slightly simpler since we may take $v(t)\equiv0$, so we leave it to the reader.

\begin{rem}\label{Rem5.igo-go} To the best of our knowledge the idea to split the
 solution $u$ into a sum \eqref{5.split} of regular and small components which are
 constructed using the previously obtained splitting into decaying
  and regular, but exponentially growing components has been suggested in \cite{ZCPAA2004} for the
  study of cubic non-autonomous damped wave equations. It has been widely
  used later in various modifications, see e.g., \cite{CP,YangSun} and, in particular, in \cite{SZ-UMN}
  which is most
   close to our work and where the finite energy solutions of quintic wave equation have been studied. However,
   only the smallness in mean for the $\tilde v$ component has been obtained there
   $$
   \int_\tau^t\|\tilde v(\tau)\|_{L^{12}}^4+\|\xi_{\tilde v}(\tau)\|_{\E}\,d\tau\le C_\delta+\delta(t-\tau)
   $$
   for all $\delta>0$. The extra  term  $C_\delta$ is not dangerous for finite energy solutions
   where the second integral in the right-hand side of \eqref{sm.est-good} is absent, but
    is {\it not acceptable} in our case since it leads to the dependence of the
    constant $C''$ in \eqref{est.w.ddis} on $\delta=\delta(R)$ and, as a result, we will
     be unable to make the constant $C'R^{-1}$ small no matter how big $R$ is.
     \par
     Thus, the result of Lemma \ref{Lem8.usplit} is an essential and useful improvement of the
      scheme which has been somehow overseen in the previous papers and which has an independent interest.
\end{rem}

\section{Attractors and concluding remarks}\label{s6}
The aim of this section is to  build up the attractor theory for the SS-solutions of the damped
 quintic wave equation \eqref{eq.main} and to discuss its natural generalizations.
 We restrict ourselves to consider the autonomous case only, so assumption \eqref{5.g}
  is assumed to be satisfied (see Remark \ref{Rem6.non} for a brief discussion of the non-autonomous case).
  In this case, due to Theorem \ref{Th4.dis}, equation \eqref{eq.main} defines
   a dissipative semigroup $S(t)$, $t\ge0$, in the uniformly local phase space $\E_b$:
\begin{equation}\label{6.sem}
S(t)\xi_0:=\xi_u(t),\ \ t\ge0,\ \ S(t):\E_b\to\E_b,
\end{equation}
where $u(t)$ is a uniquely defined SS-solution of problem \eqref{eq.main} with
 the initial data $\xi_0\in\E_b$. Moreover, estimate \eqref{4.dis} now reads
 \begin{equation}
\|S(t)\xi_0\|_{\E_b}\le Q(\|\xi_0\|_{\E_b})e^{-\beta t}+Q(\|g\|_{L^2_b})
 \end{equation}
and guarantees the existence of an absorbing ball for the solution semigroup in $\E_b$.
 However, in contrast to the case of bounded domains or/and finite energy solutions, the solution
 semigroup $S(t)$ does not possess in general a compact global attractor in the space $\E_b$, so the
  concept of the so-called {\it locally compact} attractor is naturally used instead,
   see \cite{MirZel} and references therein for more details.
\par
We recall that, by definition, a set $\Cal A$ is a locally compact global attractor of a semigroup $S(t)$ acting
 in the uniformly local space $\E_b$ if
 \par
 1. $\Cal A$ is bounded in $\E_b$ and is compact in $\E_{loc}$. The latter means that for any ball $B^R_{x_0}$,
 the restriction $\Cal A\big|_{B^R_{x_0}}$ is compact in $\E(B^R_{x_0})$.
 \par
 2. $\Cal A$ is strictly invariant, i.e. $S(t)\Cal A=\Cal A$ for all $t\ge0$.
 \par
 3. $\Cal A$ attracts the images of all bounded in $\E_b$ sets in the topology of $\E_{loc}$. This means that, for every
 bounded set $B\subset\E_b$ and every neighbourhood $\Cal O(\Cal A)$ of the attractor $\Cal A$ in
  the topology of $\E_{loc}$, there exists $T=T(B,\Cal O)$ such that
  $$
  S(t)B\subset \Cal O(\Cal A),\ \ \text{for}\ \ t\ge T.
  $$
  The existence of such an attractor can be verified using the following standard attractor's existence result.

  \begin{prop}\label{Prop6.atr} Let the semigroup $S(t):\E_b\to\E_b$ be continuous for every fixed $t\ge0$ in the topology
   of $\E_{loc}$ and possesses a bounded in $\E_b$ and compact in $\E_{loc}$ attracting set $\Cal B$. Then
   there exists a (locally compact global) attractor $\Cal A\subset\Cal B$ for this semigroup. Moreover,
    this attractor is generated by all bounded trajectories of the semigroup $S(t)$ defined
     for all $t\in\R$:
\begin{equation}\label{6.rep}
\Cal A=\Cal K\big|_{t=0},
\end{equation}
where
\begin{equation}\label{6.ker}
\Cal K:=\{\xi_u\in L^\infty(\R,\E_b),\ \ \xi_u(t+\tau)=S(t)\xi_u(\tau),\ \ \tau\in\R,\ \ t\in\R_+\}
\end{equation}
is the set of all bounded complete trajectories of the semigroup $S(t)$ (the  kernel of $S(t)$
 in the terminology of Chepyzhov and Vishik, see \cite{CV}).
   \end{prop}
For the proof of this criterion, see e.g., \cite{BV,CV}.
\par
Applying this criterion, to the solution semigroup $S(t)$ generated by equation \eqref{eq.main}, we
get the following result.

\begin{theorem}\label{Th6.attr} Let the nonlinearity $f$ satisfy \eqref{4.f} and the external force $g$ enjoy \eqref{5.g}.
 Let also the semigroup $S(t)$ associated with equation \eqref{eq.main} be defined by \eqref{6.sem}. Then this
  semigroup possesses a (locally compact) global attractor $\Cal A$ which is a bounded set of
   $\E^{\alpha}_b$ for $\alpha\in(0,\frac25]$. Moreover, the representation formula \eqref{6.rep} holds and
    the set $\Cal K$ of all bounded complete solutions of \eqref{eq.main} possesses the following estimate:
    \begin{equation}\label{6.ka}
\|\xi_u\|_{L^\infty(\R,\E_b^\alpha)}+\|u\|_{L^4_b(\R,H^{\alpha,12}_b)}\le Q(\|g\|_{L^2_b}),\ \ \xi_u\in\Cal K,
    \end{equation}
where the function $Q$ is independent of $\xi_u\in\Cal K$ and $\alpha\le\frac25$.
\end{theorem}
\begin{proof} Indeed, the continuity of the operators $S(t)$ in $\E_{loc}$ is an immediate
 corollary of estimate \eqref{4.cone}. In addition, estimates \eqref{5.vdec} and \eqref{5.w-dis} guarantee that
 the set
 $$
 \Cal B_R:=\{\xi\in\E_b^\alpha,\ \ \|\xi\|_{\E_b^\alpha}\le R\}
 $$
 is an attracting set for $S(t)$ if $R=R(\|g\|_{L^2_b})$ is large enough and $\alpha\in(0,\frac25]$.
  Obviously this set is bounded and closed in $\E_b$ and is compact in $\E_{loc}$. Thus, the existence
   of an attractor $\Cal A\subset\Cal B_R$ follows from Proposition \ref{Prop6.atr}.
   Finally, estimate \eqref{6.ka} is also an immediate corollary of \eqref{5.w-dis} and the theorem is proved.
\end{proof}
As usual, the further regularity of the attractor $\Cal A$ can be obtained by the standard bootstrapping arguments and is restricted by
 the regularity of $f$ and $g$ only. In particular, under our assumptions we may guarantee that
 the solutions are $\E^1_b$-regular.

\begin{theorem}\label{Th6.a-sm} Let the assumptions of Theorem \ref{Th6.attr} hold. Then the attractor
 $\Cal A$ of problem \eqref{eq.main} constructed in the previous theorem is a bounded set of $\Cal E^1_b$.
 Moreover, problem \eqref{eq.main} is globally well-posed in the higher energy space $\E^1_b$ and
  the following dissipative estimate holds:
  \begin{equation}\label{6.dis1}
  \|\xi_u(t)\|_{\E^1_b}\le Q(\|\xi_u(0)\|_{\E^1_b})e^{-\beta t}+Q(\|g\|_{L^2_b}),
  \end{equation}
  where the positive constant $\beta$ and monotone function $Q$ are independent of $u$, $g$ and $t$.
\end{theorem}
\begin{proof} Actually, one extra step of bootstrapping is enough to improve the regularity
of the attractor from
 $\E^\alpha_b$ ($\alpha>\frac18$) to $\E^1_b$. Moreover, the non-linear decomposition \eqref{5.eq.v}
  and \eqref{5.eq.w} is no more necessary and much simpler linear splitting works. Namely,
  let now $u(t)=v(t)+w(t)$ where, in contrast to Section \ref{s5}, the function $v$ solves
   the linear equation
\begin{equation}\label{6.v}
  \Dt v+\gamma\Dt v+(1-\Dx)v=0,\ \ \xi_{v}\big|_{t=0}=\xi_u\big|_{t=0}
\end{equation}
and the smooth component $w$ solves
\begin{equation}\label{6.w}
\Dt w+\gamma\Dt w+(1-\Dx)w=g-f(u),\ \ \xi_w\big|_{t=0}=0.
\end{equation}
Indeed, applying estimate \eqref{3.al-weight} to equation \eqref{6.v} and taking
 the supremum over $x_0\in\R^3$, we arrive at the decaying estimate
 \begin{equation}\label{6.vdec}
 \|\xi_v(t)\|_{\E^\alpha_b}\le C\|\xi_u(0)\|_{\E^\alpha_b}e^{-\beta t}.
 \end{equation}
On the other hand, as not difficult to verify using the growth restriction of $f$ together with the
 Sobolev embedding theorem
 and proper interpolation inequalities,
 \begin{equation}\label{6.int}
\|f(u)\|_{L^1(t,t+1;H^1_b)}\le C\(1+\|u\|^4_{L^4(t,t+1;H^{\alpha,12}_b)}\)
\|\xi_u\|_{L^\infty(t,t+1;\E^\alpha_b)},
 \end{equation}
 where $\alpha>\frac18$. Therefore, we may apply estimate \eqref{3.al-weight} with $\alpha=1$ to equation \eqref{6.w} and obtain
  with the help of estimate \eqref{5.aldis} and the trick with function $\theta$ described
   at the beginning of the proof of Proposition \ref{Prop5.wdiv} that
 \begin{equation}\label{6.w-dis}
\|\xi_w(t)\|_{\E^1_b}\le Q(\|\xi_u(0)\|_{\E^\alpha_b})e^{-\beta t}+Q(\|g\|_{L^2_b}).
 \end{equation}
 Estimates \eqref{6.vdec} and \eqref{6.w-dis} guarantee that the attractor
  $\Cal A$ is a bounded set in $\E^1_b$. Finally, in order to get estimate \eqref{6.dis1}, it is
   sufficient to take $v\equiv0$ and repeat the derivation of \eqref{6.w-dis}. Thus, the theorem is proved.
\end{proof}
\begin{rem} Arguing in a standard way (e.g., using the energy method, see \cite{ball,MS}) one can easily show that
 the attractor $\Cal A$ is a compact set in $\E^1_{loc}$. However, the inclusion $\Cal A\subset\E^{1+\eb}_b$
  for some positive $\eb$
  is not true in general if $g\in L^2_b$ only (we need more regularity of $g$ to get this result).
\end{rem}
\begin{rem}\label{Rem6.ent} Since $H^2_b\subset C_b$, the growth rate of $f$ is no more important if
Theorem \ref{Th6.a-sm} is proved (we may just cut off the non-linearity $f$ outside of the attractor),
 so all further results about the properties of the attractor obtained for energy subcritical
  (sub-cubic) growth rate of the non-linearity are automatically extended to the quintic case.
\par
In particular, as known (see e.g., \cite{MirZel} and references therein), in contrast
 to the case of bounded domains, locally compact attractors in uniformly local spaces usually
  have infinite Hausdorff and fractal dimensions. By this reason, one usually replaces
   the dimension estimates by the proper estimates of Kolmogorov's $\eb$-entropy.
   \par
   We recall that if $K$ is a compact set in a metric space $X$, then by Hausdorff criterion it can be covered by
   finitely many of $\eb$-balls for any $\eb>0$. Let $N_\eb(K,X)$ be the minimal number of such balls.
   Then, by definition, the Kolmogorov's entropy of $K$ in $X$ is the following number:
   $$
   \Bbb H_\eb(K,X):=\log_2N_\eb(K,X),
   $$
see \cite{KoT93} for details. In particular, the case of finite fractal dimension corresponds to the estimate
$$
\Bbb H_\eb(K,X)\le d_f(K)\log_2\frac1\eb+o(\log_2\frac1\eb).
$$
Since the attractor $\Cal A$ is not compact in $\E_b$, but only in $\E_{loc}$, it is natural introduce
 the quantities $\Bbb H_\eb(\Cal A\big|_{B^R_{x_0}},\E(B^R_{x_0}))$ and study their
 dependence on two parameters $R$ and $\eb$. It is known, see \cite{MirZel,ZCPAM} and references therein that,
 for many classes of dissipative PDEs in unbounded domains, these quantities possess the following universal
  estimates:
  \begin{equation}\label{6.Aent}
\Bbb H_\eb(\Cal A\big|_{B^R_{x_0}},\E(B^R_{x_0}))\le C(R+\log_2\frac1\eb)^3\log_2\frac1\eb,
  \end{equation}
  where $C$ is independent of $\eb$ and $R$ and $\eb>0$
   (the exponent $3$ here is the space dimension $x\in\R^3$) and these estimates are sharp, see \cite{EMZ1}.
   \par
   For the case of damped wave equation \eqref{eq.main} with sub-cubic growth rate of the
    non-linearity $f$ they are obtained in \cite{Zhyp} (see also \cite{MPS}). As explained above,
    the result of Theorem \ref{Th6.a-sm} allows us to extend this estimate to the case of
     quintic wave equations in $\R^3$.
\end{rem}

\begin{rem} Similarly to the case of bounded domains, we may introduce {\it exponential attractors}
 for the problem \eqref{eq.main}. Since the global attractor is already infinite-dimensional,
 the properly defined exponential attractor must be also infinite dimensional, so in order
  to control its size it is natural (following \cite{EMZ}) to use universal entropy
   estimates \eqref{6.Aent}. Namely, by definition,  $\Cal M$ is an exponential attractor
    for the semigroup $S(t):\E_b\to\E_b$ if
    \par
    1. The set $\Cal M$ is bounded in $\E_b$ and compact in $\E_{loc}$.
    \par
    2. The set $\Cal M$ enjoys the universal entropy estimates \eqref{6.Aent}.
    \par
    3. The set $\Cal M$ is semi-invariant, i.e., $S(t)\Cal M\subset\Cal M$ for $t\ge0$.
    \par
    4. The exponential attraction property
\begin{equation}
    \dist_{\E_b}(S(t)B,\Cal M)\le Q(\|B\|_{\E_b})e^{-\beta t}
\end{equation}
    holds for every bounded set $B$ in $\E_b$. Here $\dist_{\E_b}(X,Y)$ stands for the
    non-symmetric Hausdorff distance between sets $X$ and $Y$ in $\E_b$ and the positive
     constant $\beta>0$ and monotone function $Q$ are independent of $B$ and $t$.
\par
The existence of such an object in the case of reaction-diffusion equations in unbounded domains in
uniformly local phase spaces is verified in \cite{EMZ}. The estimates for differences between solutions
 for damped wave equations allows to expect the same result to be true for equation \eqref{eq.main} as well.
 We return to this question somewhere else.
 \par
 It also worth to emphasize that the attraction to the exponential attractor holds
 in a {\it uniform}  topology of the space $\E_b$ and this is one of  extra advantages of
 the exponential attractors approach. It is well-known, that for
  the global attractor $\Cal A$ we have the attraction property in a local topology
  of $\E_{loc}$ only (there are natural examples where
   the attraction property in $\E_b$ fail, see \cite{MirZel} for more details).
\end{rem}
\begin{rem}\label{Rem6.non} To conclude we note that the autonomous case of equation \eqref{eq.main} has been chosen
just for simplicity. All of the asymptotic smoothing results hold for general non-autonomous
 external forces $g(t)$ as well if we pose some extra regularity assumptions on $g$, for instance,
 $$
 g\in L^1_b(\R,H^1_b) \ \ \text{or}\ \ \ g\in W^{1,1}_b(\R,L^2_b).
 $$
The only difference is that we will need to consider instead of global attractors their proper
 generalizations to the non-autonomous case (e.g., uniform or pull-back attractors). We also expect that
 most part of the results obtained in \cite{SZ-UMN} for the case of periodic boundary conditions can be
  naturally extended to the case of infinite-energy solutions in the whole space. We return to this problem
  somewhere else.
\end{rem}

\appendix
\section{Estimates in fractional Sobolev spaces}\label{sB}
In this Appendix we discuss the estimates in fractional Sobolev spaces which are necessary
 to treat the nonlinear term $f(u)$ in equation \eqref{eq.main}. We start with the
  corollary of Kato-Ponce inequality which is proved in \cite{SZ-UMN}.
\begin{prop}\label{Lem.v4wHa}
	Let $\alpha\in(0,2/5]$ and let the  functions $v$ and $w$ be such that
\begin{equation}\label{A2.vw}
	v\in L^{12}(\R^3)\cap H^1(\R^3),\quad w\in H^{\alpha,12}(\R^3)\cap H^{1+\alpha}(\R^3).
	\end{equation}
Assume also that the function $h\in C^1(\R)$,  satisfies $h(0)=0$ and
\begin{equation}\label{ap.h}
	|h'(v)|\leq C(1+|v|^3)
	\end{equation}
for some constant $C>0$ and all $v\in\R$. Then $h(v)w\in H^{\alpha}(\R^3)$ and the following estimate holds:
\begin{equation}\label{a.est.main}
	\|h(v)w\|_{H^\alpha}\le C_\alpha\left(1+\|v\|_{L^{12}}^{4-\alpha}\right)\|v\|_{H^1}^\alpha
	\|w\|^{1-\alpha}_{H^{1+\alpha}}\|w\|^{\alpha}_{H^{\alpha,12}},
	\end{equation}
for some positive constant $C_\alpha$.
\end{prop}
We need the analogue of this estimate for a  bounded domain $V\subset\R^3$ (used
in the paper  for $V=B^R_{x_0}$ only).
\begin{cor}\label{CorA.V} Let $V$ be a bounded domain in $\R^3$ with smooth boundary and let the assumptions of
Proposition \ref{Lem.v4wHa} hold. Then the following estimate holds:
\begin{multline}\label{A.V}
\|h(v)w\|_{H^\alpha(V)}\le \\\le C_\alpha\left(1+\|v\|_{L^{12}(V)}^{4-\alpha}\right)
\|v\|_{H^1(V)}^\alpha
	\|w\|^{1-\alpha}_{H^{1+\alpha}(V)}\|w\|^{\alpha}_{H^{\alpha,12}(V)}.
\end{multline}
\end{cor}
Indeed, this is an immediate corollary of \eqref{a.est.main}, the definition of the spaces $H^\alpha(V)$ and
 the existence of an extension operator from $V$ to $\R^3$.
 \par
 We now turn to the estimates of $f(u)-f(v)$ which are crucial for our proof
  of asymptotic smoothing property.

\begin{cor}\label{CorA.f} Let $f\in C^2$ satisfy assumptions \eqref{4.f} and $f'(0)=0$. In addition, let the
functions $u$ and $v$ satisfy \eqref{A2.vw}. Assume also that the cut-off function
 $\psi\in C^\infty_0(\R^3)$ be such that $\psi(x)\equiv1$ for $x\in B^1_0$ and $\psi(x)\equiv0$ for $x\notin B^{3/2}_0$.
\par
Let finally $\psi_{R,x_0}(x):=\psi(R^{-1}(x-x_0))$ for some $R>1$ and $x_0\in\R$.
 Then the following estimate holds:
\begin{multline}\label{A.mf}
\|\psi_{R,x_0}(f(u)-f(v))\|_{H^\alpha}\le
C\(1+\|u\|_{L^{12}(B^{2R}_{x_0})}+\|v\|_{L^{12}(B^{2R}_{x_0})}\)^{4-\alpha}\times\\\times
\(\|u\|_{H^1(B^{2R}_{x_0})}+\|v\|_{H^1(B^{2R}_{x_0})}\)^\alpha
\|\psi_{R,x_0}(u-v)\|_{H^{1+\alpha}}^{1-\alpha}\|\psi_{R,x_0}(u-v)\|_{H^{\alpha,12}}^\alpha,
\end{multline}
where the constant $C$ is independent of $R$ and $x_0$.
\end{cor}
\begin{proof} Indeed, using the analogue of estimates \eqref{2.V} for the scaled functions $\psi_{R,x_0}$,
we get
\begin{multline}\label{A.dif}
\|\psi_{R,x_0}(f(u)-f(v))\|_{H^\alpha}\le C\|\psi_{R,x_0}(f(u)-f(v))\|_{H^\alpha(B^{2R}_{x_0})}=\\=
\|\psi_{R,x_0}\int_0^1 f'(\lambda u+(1-\lambda v))(u-v)\,d\lambda\|_{H^\alpha(B^{2R}_{x_0})}\le\\\le
C\int_0^1\|f'(\lambda u+(1-\lambda)v)\psi_{R,x_0}(u-v)\|_{H^\alpha(B^{2R}_{x_0})}\,d\lambda.
\end{multline}
Note that the function $h(u)=f'(u)$ satisfies all assumptions of Proposition \ref{Lem.v4wHa}, so
 we may use \eqref{A.V} to estimate the right-hand side of \eqref{A.dif}. Using also that,
 by the definition of the space $H^\alpha(B^{2R}_{x_0})$,
$$
\|\psi_{R,x_0}u\|_{H^\alpha(B^{2R}_{x_0})}\le \|\psi_{R,x_0}u\|_{H^\alpha(\R^3)},
$$
we get the desired estimate and finish the proof of the corollary.
\end{proof}
We conclude this section by stating one more useful corollary of the key estimate \eqref{A.V}.
\begin{cor}\label{CorA.split} Let the assumptions of Corollary \eqref{CorA.f} hold and let, in addition,
\begin{equation}\label{A.uvw}
u(x)-v(x)=w_1(x)+w_2(x)
\end{equation}
for some functions $w_1$ and $w_2$ satisfying \eqref{A2.vw}. Then the following estimate holds:
\begin{multline}\label{A.mfs}
\|\psi_{R,x_0}(f(u)-f(v))\|_{H^\alpha}\le
C\(1+\|u\|_{L^{12}(B^{2R}_{x_0})}+\|v\|_{L^{12}(B^{2R}_{x_0})}\)^{4-\alpha}\times\\\times
\(\|u\|_{H^1(B^{2R}_{x_0})}+\|v\|_{H^1(B^{2R}_{x_0})}\)^\alpha\times\\\times
\(\|\psi_{R,x_0}w_1\|_{H^{1+\alpha}}^{1-\alpha}\|\psi_{R,x_0}w_1\|_{H^{\alpha,12}}^\alpha+
\|\psi_{R,x_0}w_2\|_{H^{1+\alpha}}^{1-\alpha}\|\psi_{R,x_0}w_2\|_{H^{\alpha,12}}^\alpha\),
\end{multline}
where the constant $C$ is independent of $R$ and $x_0$.
\end{cor}
Indeed, to verify \eqref{A.mfs}, we just need to put \eqref{A.uvw} into the right-hand side of \eqref{A.dif}
 and apply estimate \eqref{A.V} to every of two obtained terms separately.

\section{Proof of commutator estimates}\label{sA}
In this Appendix we give the brief proof of estimates \eqref{2.com1} and \eqref{2.com2} stated
 in Proposition \ref{Prop2.com}, see also \cite{S.ieqwv1h} for the analogous proof
 in the particular case $p=2$. To this end, we will use the following formula for fractional powers:
 \begin{equation}
(1-\Dx)^\alpha u:=\frac1{\Gamma(-\alpha)}\int_0^\infty (e^{-t(1-\Dx)} u-u)\frac{dt}{t^{1+\alpha}}
 \end{equation}
 for $\alpha\in(0,1)$, see e.g., \cite{Tri}. Remind that in our case $\alpha=s/2\in(0,1/2)$. Let now
  $u\in C_0^\infty(\R^n)$ and $\psi$ be either also from
  $C_0^\infty(\R^n)$ or $\psi=\phi_{\eb,x_0}$. Then
\begin{multline*}
  \psi(1-\Dx)^{\alpha}u-(1-\Dx)^{\alpha}(\psi u)=\\=\frac1{\Gamma(-\alpha)}
  \int_0^\infty U(t)\frac{dt}{t^{1+\alpha}}=\frac{-1}{\Gamma(1-\alpha)}
  \int_0^\infty\frac{\Dt U(t)}{t^\alpha}\,dt,
\end{multline*}
where the function $U(t):=\psi e^{-t(1-\Dx)} u-e^{-t(1-\Dx)}(\psi u)$
solves the following parabolic problem:
\begin{equation}\label{2.2heat}
\begin{cases}
\Dt U+(1-\Dx)U=-2\Nx\psi\Nx\bar u-\Dx\psi\bar u:=h_\psi(t),\ \ U\big|_{t=0}=0,\\
\Dt\bar u+(1-\Dx)\bar u=0,\ \ \bar u\big|_{t=0}=u.
\end{cases}
\end{equation}
Note also that
$$
\|h_\psi(t)\|_{L^p}\le C_{\psi,x_0}\|\bar u(t)\|_{W^{1,p}_{\phi_{\eb,x_0}}}
$$
and in the case $\psi=\phi_{\eb,x_0}$ we have $C_{\psi,x_0}=C|\eb|$. Moreover, applying
 the weighted parabolic smoothing property to the second equation of \eqref{2.2heat}, we arrive at
 $$
t^{1/2}\|h_\psi(t)\|_{L^p}\le Ce^{-\kappa t}\|u\|_{L^p_{\phi_{\eb,x_0}}}
 $$
 for some positive $\kappa$ (the weighted smoothing property follows immediately from the
  classical non-weighted one and the trick with multiplication operator $T_{\eb,x_0}$). Thus,
  for every $\delta\in(0,1)$, we have
$$
\int_0^\infty e^{\kappa t}\|h_{\psi}(t)\|_{L^p}^{2-\delta}\,dt\le C_{p,\delta}
C_{\psi,x_0}^{2-\delta}\|u\|_{L^p_{\phi_{\eb,x_0}}}^{2-\delta}.
$$
It is well-known that the heat equation
$$
\Dt W+(1-\Dx)W=h(t),\ \ W\big|_{t=0}=0
$$
possesses the following anisotropic $L^q(L^p)$-regularity estimate
$$
\|\Dt W\|_{L^q(R_+,L^p(\R^3))}+\|W\|_{L^q(\R_+,H^{2,p}(\R^3))}\le C_{p,q}\|h\|_{L^q(\R_+,L^p(\R^3))}
$$
for all $1<p,q<\infty$, see e.g., \cite{LMR}.
Applying this regularity result to the first equation
of \eqref{2.2heat}, we arrive at
$$
\int_0^\infty e^{\kappa t}\|\Dt U(t)\|_{L^p}^{2-\delta}\,dt\le C_{p,\delta}
C_{\psi,x_0}^{2-\delta}\|u\|_{L^p_{\phi_{\eb,x_0}}}^{2-\delta}
$$
and finally
\begin{multline*}
\|\psi(1-\Dx)^{\alpha}u-(1-\Dx)^{\alpha}(\psi u)\|_{L^p}\le
C\int_0^\infty\|\Dt U(t)\|_{L^p}\frac{dt}{t^\alpha}\le\\\le
 \(\int_0^\infty e^{\kappa t}\|\Dt U(t)\|_{L^p}^{2-\delta}\,dt\)^{\frac1{2-\delta}}
 \(\int_0^\infty\frac {e^{-\kappa_1t}dt}{t^{\alpha\frac{2-\delta}{1-\delta}}}\)^{\frac{1-\delta}{2-\delta}}\le\\\le
 C_pC_{\psi,x_0}\|u\|_{L^p_{\phi_{\eb,x_0}}}
\end{multline*}
if $\delta>0$ is small enough that $\alpha\frac{2-\delta}{1-\delta}<1$ and the
commutator estimates are proved.


\end{document}